\documentclass[a4paper,10pt,intlimits,sumlimits]{amsart}
\usepackage[utf8]{inputenc}

\usepackage{amsmath}
\usepackage{amsthm}
\usepackage{amstext}
\usepackage{amssymb}
\usepackage{amsfonts}
\usepackage{mathrsfs}
\usepackage{mathtools}
\usepackage{enumerate}

\usepackage{esint}

\usepackage{bbm} 
\usepackage{bm}

\usepackage{pgf}
\usepackage{tikz}
\usepackage{pgfplots}
\usetikzlibrary{patterns,calc,decorations.pathreplacing}

\usepackage{float}
\usepackage{graphicx}
\usepackage{wrapfig}

\usepackage[pdftex,bookmarks,colorlinks,breaklinks]{hyperref}  
\definecolor{dullmagenta}{rgb}{0.4,0,0.4}   
\definecolor{darkblue}{rgb}{0,0,0.4}
\definecolor{darkgreen}{rgb}{0,0.4,0}
\hypersetup{linkcolor=darkblue,citecolor=blue,filecolor=dullmagenta,urlcolor=darkblue} 

\newcommand{\mc}{\mathcal}
\newcommand{\ms}[1]{\mathscr{#1}}

\newcommand{\mbb}[1]{\mathbb{#1}}
\newcommand{\mf}[1]{\mathfrak{#1}}

\newcommand{\map}[3]{#1 \colon #2 \rightarrow #3}

\newcommand{\R}{\mathbb{R}} 
\newcommand{\C}{\mathbb{C}}
\newcommand{\E}{\mathbb{E}} 
\newcommand{\Q}{\mathbb{Q}} 
\newcommand{\N}{\mathbb{N}} 
\newcommand{\Z}{\mathbb{Z}} 
\newcommand{\I}{\mathbb{I}}

\newcommand{\RR}{\R}  
\newcommand{\CC}{\C}

\newcommand{\II}{\I}

\newcommand{\dd}{\mathrm{d}}
\DeclareMathOperator{\pv}{p.v.}

\renewcommand{\bar}{\overline} 
\renewcommand{\tilde}{\widetilde} 

\newcommand{\FT}{\widehat}

\renewcommand{\phi}{\varphi}
\renewcommand{\epsilon}{\varepsilon}

\newcommand{\1}{\mathbbm{1}}

\DeclareMathOperator{\spt}{spt}

\newcommand{\Lin}{\mc{L}} 

\DeclareMathOperator{\Mod}{Mod} 
\DeclareMathOperator{\Tr}{Tr} 
\DeclareMathOperator{\Dil}{Dil} 

\newcommand{\VCarl}{\ms{V}}
\newcommand{\Carl}{\ms{C}}
\newcommand{\Conv}{\mathrm{Conv}} 
\newcommand{\Mult}{\mathrm{Mult}} 

\newcommand{\Emb}{\mathrm{E}}
\newcommand{\AEmb}{\mathrm{A}}
\newcommand{\MEmb}{\mathrm{M}}

\newcommand{\Err}{\mathrm{Err}} 

\newcommand{\sL}{\hbox{\raisebox{0.06em}-}\kern-0.45emL} 
\newcommand{\OX}{\mathbb{X}} 
\newcommand{\OB}{\mathbb{B}} 
\newcommand{\OS}{\mathbb{S}} 

\newcommand{\LS}{\mathfrak{L}} 
\newcommand{\RS}{\mathfrak{R}} 
\newcommand{\FS}{\mathfrak{S}} 

\newcommand{\mT}{\mf{T}} 
\newcommand{\mD}{\mf{D}} 
\newcommand{\TT}{\mbb{T}} 
\newcommand{\DD}{\mbb{D}} 

\newcommand{\Bor}{\ms{B}} 

\newcommand{\Sch}{\ms{S}} 

\newcommand{\sm}{\setminus}

\newtheorem{thm}{Theorem}
\newtheorem*{thm*}{Theorem}

\newtheorem{defn}[thm]{Definition}
\newtheorem*{defn*}{Definition}

\newtheorem{prop}[thm]{Proposition}
\newtheorem*{prop*}{Proposition}

\newtheorem{cor}[thm]{Corollary}
\newtheorem*{cor*}{Corollary}

\newtheorem{lem}[thm]{Lemma}
\newtheorem*{lem*}{Lemma}

\newtheorem{rmk}[thm]{Remark}
\newtheorem*{rmk*}{Remark}

\numberwithin{equation}{section}
\numberwithin{thm}{section}

\begin{document}

\title{Variational Carleson operators in UMD spaces}
\date{\today}

\author[A. Amenta]{Alex Amenta}
\address{\noindent Mathematisches Institut \newline \indent Universit\"at Bonn, Bonn, Germany}
\email{amenta@math.uni-bonn.de}

\author[G. Uraltsev]{Gennady Uraltsev}
\address{\noindent Department of Mathematics \newline \indent Cornell University, Ithaca, NY, USA}
\email{guraltsev@math.cornell.edu}



\begin{abstract}
  We prove $L^p$-boundedness of variational Carleson operators for functions valued in intermediate UMD spaces.
  This provides quantitative information on the rate of convergence of partial Fourier integrals of vector-valued functions.
  Our proof relies on bounds on wave packet embeddings into outer Lebesgue spaces on the time-frequency-scale space $\R^3_+$, which are the focus of this paper.
\end{abstract}


\maketitle
\setcounter{tocdepth}{2}
\tableofcontents

\section{Introduction}
\label{sec:intro}

Consider a Schwartz function $f \in \Sch(\R;X)$ taking values in a Banach space $X$.
For all frequencies $\xi \in \R$, the partial Fourier integral of $f$ up to $\xi$ is
\begin{equation*}
  \Carl_{\xi}f(x) := \int_{-\infty}^{\xi} \hat{f}(\eta) e^{2\pi i \eta x} \, \dd\eta.
\end{equation*}
By Fourier inversion, $f$ can be recovered as the pointwise limit of partial Fourier integrals
\begin{equation}
  \label{eq:Fourier-convergence}
  \lim_{\xi \to \infty} \Carl_{\xi} f(x) = f(x) \qquad \forall x \in \R.
\end{equation}
This limit holds for all Schwartz functions, but when $f$ is merely assumed to be in the Bochner space $L^p(\R;X)$ for some $p \in (1,\infty)$, the pointwise convergence of partial Fourier integrals is a difficult matter.
When the functions under consideration are scalar-valued, i.e. when $X = \C$, the Carleson--Hunt theorem says that the limit \eqref{eq:Fourier-convergence} exists for almost all $x \in \R$.
Various proofs of this theorem are known, all using some form of time-frequency analysis; see \cite{lC66,cD15,cF73,rH68,LT00,cT06}.
More generally, when $X$ is a Banach space with the \emph{intermediate UMD} property (defined in Section \ref{sec:UMD}), the convergence \eqref{eq:Fourier-convergence} still holds for almost all $x \in \R$, as shown by Hyt\"onen and Lacey \cite{HL13}.

The almost-everywhere validity of the limit \eqref{eq:Fourier-convergence} for functions in $L^p(\R;X)$ follows from the $L^p$-boundedness of the \emph{Carleson maximal operator} $\Carl_{*}$, defined on $f \in \Sch(\R;X)$ by
\begin{equation*}
  \Carl_{*} f(x) := \sup_{\xi\in \R} \|\Carl_{\xi} f(x)\|_X = \|\xi\mapsto\Carl_{\xi} f(x)\|_{L^\infty(\RR;X)} \qquad \forall x \in \R.
\end{equation*}
However, this argument does not provide any information on the rate of convergence, and it relies on having already established the convergence for Schwartz functions.
To get quantitative information without need for an \emph{a priori} result on a dense subclass, one can consider \emph{variational operators}.
For $r \in [1,\infty)$, the $r$-variational Carleson operator is defined by
\begin{equation}
  \label{eq:var-carl-defn}
  \begin{aligned}
  \VCarl^{r}_{*} f(x) &:=  \|\xi \mapsto \Carl_{\xi} f(x)\|_{V^r(\RR;X)},
\end{aligned}
\end{equation}
where the $r$-variation of a path $\map{u}{\RR}{X}$ in a Banach space $X$ is
\begin{equation*}
  \|u\|_{V^r(\RR;X)} := \sup_{\mf{c}\in\Delta } \Big( \sum_{j=0}^\infty \|u(\mf{c}_{j+1}) - u(\mf{c}_{j}) \|_X^r \Big)^{1/r}
\end{equation*}
with supremum taken over all increasing sequences $(\mf{c}_j)_{j\in\N}$ in $\R$; here
\begin{equation*}
  \Delta:=\{\mf{c}\in\R^{\N} \colon \mf{c}_{j}\leq\mf{c}_{j+1}, j\in\N\}
\end{equation*}
is the infinite simplex.

In the scalar-valued setting, Oberlin, Seeger, Tao, Thiele, and Wright showed that the $r$-variational Carleson operator $\VCarl_*^r$ is $L^p$-bounded for all $r' < p < \infty$, provided that $r > 2$ \cite{OSTTW12}.
In this article we extend their result to intermediate UMD spaces (defined in Section \ref{sec:UMD}). 

\begin{thm}
  \label{thm:main}
  Let $X$ be an $r_0$-intermediate UMD Banach space for some $r_0 \in [2,\infty)$.
  Then
  \begin{equation*}
    \| \VCarl_{*}^{r} f\|_{L^p(\RR)} \lesssim_{p,r,X} \|f\|_{L^p(\RR;X)}
  \end{equation*}
  for all $r_0 < r < \infty$ and $(r/(r_0-1))' < p < \infty$.
\end{thm}

When $r_0 = 2$, i.e. when $X$ is isomorphic to a Hilbert space, the condition on $p$ becomes $r' < p < \infty$, as in the scalar-valued setting.
As $r \to \infty$, the condition on $p$ tends to $1 < p < \infty$, and since $\Carl_{*} f(x) \leq \VCarl_{*}^{r} f(x)$ for all $f \in \Sch(\R;X)$ and $x \in \R$, Theorem \ref{thm:main} implies $L^p$-boundedness of the Carleson maximal operator $\Carl_*$ for all $p \in (1,\infty)$ when $X$ is intermediate UMD, which is the main result of Hyt\"onen and Lacey \cite{HL13}.

The bounds in Theorem \ref{thm:main} are a consequence of uniform bounds on linearisations of $\VCarl_*^r$, in which the pointwise $r$-variation is replaced by the pointwise $\ell^r$-norm with respect to a selection function $\mf{c}\colon \R\to \Delta$ prescribing a sequence of frequencies for each $x \in \R$.
More precisely, given a measurable function $\map{\mf{c}}{\R}{\Delta}$ (meaning that each coordinate function $\map{\mf{c}_j}{\R}{\R}$ is measurable) we define a sequence-valued function $\VCarl_{\mf{c}} f = (\VCarl_{\mf{c},j} f)_{j\in\N}$ on $\R$ by
\begin{equation*}
  \VCarl_{\mf{c},j} f(x) := \int_{\mf{c}_j(x)}^{\mf{c}_{j+1}(x)} \hat{f}(\xi) e^{2\pi i \xi x} \, \dd\xi =\Carl_{\mf{c}_{j+1}(x)}(x)-\Carl_{\mf{c}_{j}(x)}(x),
\end{equation*}
where we abuse notation and write $\mf{c}_j(x) := \mf{c}(x)_j$.
Then we have the pointwise control
\begin{equation*}
  \VCarl^{r}_{*}f(x) \leq \sup_{\map{\mf{c}}{\R}{\Delta}} \| \VCarl_{\mf{c}} f(x)\|_{l^{r}(\N;X)},
\end{equation*}
and Theorem \ref{thm:main} readily follows from the following result, which we prove.

\begin{thm}\label{thm:main-N}
  Let $X$ be an $r_0$-intermediate UMD Banach space for some $r_0 \in [2,\infty)$.
  Then for all measurable functions $\map{\mf{c}}{\R}{\Delta}$ we have 
  \begin{equation*}
    \|\VCarl_{\mf{c}} f\|_{L^p(\R;l^{r}(\N;X))} \lesssim_{p,r,X} \|f\|_{L^p(\R;X)} \qquad \forall f \in \Sch(\R;X)
  \end{equation*}
  for all $r_0 < r < \infty$ and all $(r/(r_0-1))' < p < \infty$, with implicit constant independent of $\mf{c}$.
\end{thm}

Our proof of Theorem \ref{thm:main-N} uses the framework of embeddings into outer Lebesgue spaces on the time-frequency-scale space $\R^3_+$, as in the scalar-valued setting by the second author in \cite{gU16}.
First, for all $f \in \Sch(\R;X)$ and all finitely-supported sequences $g \in c_{00}(\N;\Sch(\R;X^*))$ of $X^*$-valued Schwartz functions, the dual form to $\VCarl_{\mf{c}}$ is represented by
\begin{equation*}
  \int_{\R} \langle \VCarl_{\mf{c}}f(x); g(x) \rangle \, \dd x = \sum_{\square \in \{+,-\}} \int_{\R^3_+} \big\langle \Emb[f](\eta,y,t) ; \AEmb_{\mf{c}}^{\square}[g](\eta,y,t) \big\rangle \, \dd\eta \, \dd y \, \dd t
\end{equation*}
as the sum of two terms, each of which is the integral pairing of a \emph{wave packet embedding} $\Emb[f]$ of the function $f$ with a \emph{truncated wave packet embedding} $\AEmb_{\mf{c}}^\pm[g]$ of the sequence $g$ (see Definition \ref{defn:embeddings}). 
These integral pairings satisfy the H\"older-type bounds
\begin{equation*}
  \int_{\R^3_+} \big| \big\langle \Emb[f](\eta,y,t) ; \AEmb_{\mf{c}}^\pm[g](\eta,y,t) \big\rangle \big| \, \dd\eta \, \dd y \, \dd t
  \lesssim
  \|\Emb[f]\|_{L^p_\nu \sL^q_\mu \FS_\Theta} \|\AEmb_{\mf{c}}^\pm [g]\|_{L^{p'}_\nu \sL^{q'}_\mu \FS_\Theta^*};
\end{equation*}
the quantities on the right hand side are outer Lebesgue quasinorms of functions on $\R^3_+$ (defined in Sections \ref{sec:OLS} and \ref{sec:TFS}).
Under the given assumptions on $X$, the wave packet embedding satisfies the bounds
\begin{equation*}
  \|\Emb[f]\|_{L^p_\nu \sL^q_\mu \FS} \lesssim \|f\|_{L^p(\R;X)} \qquad \forall f \in \Sch(\R;X)
\end{equation*}
for all $p \in (1,\infty]$ and all $q > \min(p,r_0)'(r_0-1)$, as proven by the authors in \cite{AU19-2}.
Here we show the truncated wave packet embedding bounds
\begin{equation*}
  \|\AEmb_{\mf{c}}^\pm [g]\|_{L^{p'}_\nu \sL^{q'}_\mu \FS^*} \lesssim \|g\|_{L^{p'}(\R;\ell^{r'}(\N;X^*))} \qquad \forall g \in c_{00}(\N; \Sch(\R;X^*))
\end{equation*}
for suitable exponents $p$, $q$, and $r$ (see Theorem \ref{thm:A-embedding}).
We prove these by reduction to an auxiliary scalar-valued embedding map, whose boundedness was already established by the second author in \cite{gU16}.
This reduction is our Theorem \ref{thm:main-mass-dom}, which is the main technical result of the paper.
Its proof relies on variational estimates for convolution operators, along with various other technical estimates, using the UMD and cotype assumptions on $X$ (but interestingly, not the intermediacy assumption: this is because the necessary arguments are essentially of Calder\'on--Zygmund type, rather than time-frequency).

We comment briefly on previous research on variational Carleson operators.
In the scalar-valued setting, besides the initial results of Oberlin et al. \cite{OSTTW12}, there are the weighted estimates of Do and Lacey \cite{DL12-2,DL12-1}, the outer Lebesgue approach of the second author \cite{gU16}, and the sparse domination result of Di Plinio, Do, and the second author \cite{DPDU18} (which can also be proven by the helicoidal method of Benea and Muscalu \cite[\textsection 7.2]{BM18}).
For vector-valued functions, bounds have been obtained by Benea and Muscalu \cite{BM18} (for iterated Lebesgue spaces) and by the first author, Lorist, and Veraar \cite{ALV19} (for Banach function spaces with UMD concavifications).

In the general setting of UMD Banach spaces, the only result available until now was the Walsh model considered by Hyt\"onen, Lacey, and Parissis \cite{HLP14}.
In this model, the real line is replaced by the Walsh group, resulting in an idealised dyadic analysis which retains core features of the problem while eliminating pesky tail estimates.
Their theorem has essentially the same hypotheses as Theorem \ref{thm:main-N}, the only difference being that `$r_0$-intermediate UMD' is replaced with `tile-type $r$ for all $r > r_0$', which follows from our assumption.
In fact, our results follow under the assumption that $X$ is UMD and satisfies the $L^p$-$L^2$ tree orthogonality estimates of \cite[Theorem 4.5]{AU19-2} for all $p > r_0$.
It is possible that the tree orthogonality estimates are equivalent to tile-type; for now we leave this problem open and remark that there is no known way of proving either condition without assuming $r_0$-intermediate UMD.

\begin{rmk}
  Our results can be reformulated for functions on the one-dimensional torus, and deduced from the corresponding results on $\R$ by standard transference methods.
  In the scalar-valued case this is done in \cite[Appendix A]{OSTTW12}; extending this to the vector-valued case does not pose additional difficulties (for an argument in the case of Banach function spaces see \cite{ALV19}).
  Thus our results imply variational estimates for Fourier series as well as Fourier integrals.
\end{rmk}

\subsection*{Outline of the paper}
Section \ref{sec:preliminaries} contains preliminary discussions and results on technical tools needed through the paper: UMD and intermediate UMD spaces, type and cotype, $\gamma$-radonifying operators, $R$-bounds, outer Lebesgue spaces, and truncated wave packets.
This section includes important technical results, in particular Lemma \ref{lem:short-variation} (an estimate for sequences of convolution-type operators between mixed $L^2$-$\gamma$ and $L^2$-$\ell^r$ spaces for UMD spaces with cotype $r$) and Proposition \ref{prop:multiplier-wave-packet-decomposition} (truncated wave packet representation of the Fourier projection onto an interval).
Section \ref{sec:mainresult} contains the reduction of Theorem \ref{thm:main-N} to the embedding domination result of Theorem \ref{thm:main-mass-dom}, which mostly consists of applications of the preliminary results.
The heart of the paper is Section \ref{sec:mass-dom}, in which Theorem \ref{thm:main-mass-dom} is actually proven.
The final section, Section \ref{sec:BFS}, is a discussion of $r$-variational Carleson operators for functions valued in Banach function spaces.

\subsection*{Acknowledgements}
We thank Francesco Di Plinio, Christoph Thiele, and Mark Veraar for their encouragement throughout this project.
We also thank Emiel Lorist and Bas Nieraeth for useful comments on Section \ref{sec:BFS}.
The first author was supported by a Fellowship for Postdoctoral Researchers from the Alexander von Humboldt Foundation.

\subsection*{Notation}

For Banach spaces $X$ and $Y$ we let $\Lin(X,Y)$ denote the Banach space of bounded linear operators from $X$ to $Y$, and we let $\Lin(X) := \Lin(X,X)$.
For $p \in [1,\infty]$, we let $L^p(\R;X)$ denote the Bochner space of strongly measurable functions $\R \to X$ such that the function $x \mapsto \|f(x)\|_X$ is in $L^p(\R)$; for technical details see \cite[Chapter 1]{HNVW16}.
We use the notation $\langle \cdot ; \cdot \rangle$ to denote the duality pairing between a Banach space $X$ and its dual $X^*$.
For $p \in [1,\infty]$ we let $p'$ denote the conjugate exponent $p' := p/(p-1)$.
We use the Japanese bracket notation
\begin{equation*}
  \langle x \rangle := (1 + |x|^2)^{1/2} \qquad \forall x \in \R.
\end{equation*}

For each point $(y,\eta,t) \in \R^{3}_{+} = \R^2 \times \R_+$, we define the translation, modulation, and ($L^1$-normalised) dilation operators on measurable functions $\map{f}{\R}{X}$ by
\begin{equation*}
  \Tr_{y} f(z):= f(z-y) \quad \Mod_{\eta} f(z):= e^{2\pi i \eta z}f(z) \quad \Dil_{t} f(z):= t^{-1}f\Big(\frac{z}{t}\Big),
\end{equation*}
and we define the composition
\begin{equation*}
  \Lambda_{(\eta,y,t)} := \Tr_{y}\Mod_{\eta}\Dil_{t}.
\end{equation*}
For $x \in \R$ and $r > 0$ we let $B_r(x)$ denote the ball centred at $x$ with radius $r$, and we write $B_r := B_r(0)$.

%


\section{Preliminaries}
\label{sec:preliminaries}
\subsection{Notions and lemmas in Banach-valued analysis}

\subsubsection{UMD and intermediate UMD Banach spaces}\label{sec:UMD}

\begin{defn}
  A Banach space $X$ has the \emph{UMD property} if the $X$-valued Hilbert transform, defined by
  \begin{equation*}
    \mathrm{H}f(x) := \pv \int_{\R} f(x-y) \, \frac{\dd y}{y} \qquad \forall f \in \Sch(\R;X), \, x \in \R,
  \end{equation*}
  is bounded on $L^p(\R;X)$ for all $p \in (1,\infty)$.
\end{defn}

UMD stands for Unconditionality of Martingale Differences, and naturally there is an equivalent definition in these terms \cite{jB83,dB83}.
Various forms of Littlewood--Paley theory are accessible for UMD-valued functions (and their validity is equivalent to the UMD property); the particular Littlewood--Paley estimate we need appears below as Proposition \ref{prop:UMD-littewood-paley}.
Examples of UMD spaces include separable Hilbert spaces, most reflexive function spaces (Lebesgue, Sobolev, Besov, and so on), as well as reflexive non-commutative $L^p$-spaces (which are not function spaces).
Useful references on these spaces include \cite{dB01, HNVW16,gP16}.

Although UMD spaces are the most natural Banach spaces for harmonic analysis, we will need a seemingly stronger property called \emph{$r$-intermediate UMD}.
See \cite{BL76} or \cite[Appendix C]{HNVW16} for interpolation theory jargon.

\begin{defn}
  For $r \in [2,\infty)$, a Banach space $X$ is \emph{$r$-intermediate UMD} if there exists a compatible couple $(Y,H)$, where $Y$ is UMD and $H$ is a Hilbert space, such that $X$ is isomorphic to the complex interpolation space $[Y,H]_{2/r}$.
  We say $X$ is \emph{intermediate UMD} if it is $r$-intermediate UMD for some $r$.
\end{defn}

For example, an infinite-dimensional Lebesgue space $L^p$ with $p \neq 2$ (either classical or non-commutative) is $r$-intermediate UMD if and only if $r > \max(p,p')$.
One should think of $r$-intermediate UMD spaces as being slightly nicer than $L^r$ (unless $r=2$, as $2$-intermediate UMD is the same as being isomorphic to a Hilbert space, which is already the best possible situation to be in).
Rubio de Francia conjectured in \cite{RdF86} that every UMD space is actually intermediate UMD, but beyond the setting of Banach function spaces (discussed in Section \ref{sec:BFS}) the conjecture is wide open.

\subsubsection{Type, cotype, and variational estimates for convolutions}

The concepts of type and cotype, and of martingale type and cotype, are important in the geometry of Banach spaces.
We briefly introduce them here and discuss consequences for variational estimates for families of convolution operators.
In the definition below, $(\varepsilon_n)_{n=1}^N$ denotes a finite Rademacher sequence, i.e. a sequence of independent random variables on a probability space taking the values $\pm 1$ with probability $1/2$.
Much more information is to be found in \cite{HNVW17}.

\begin{defn}
  Let $X$ be a Banach space, $p \in [1,2]$, and $q \in [2,\infty]$.
  \begin{itemize}
  \item The space $X$ has \emph{type $p$} if for all finite sequences $(x_j)$ in $X$ the estimate
    \begin{equation*}
      \E \Big\| \sum_{j} \varepsilon_j x_j \Big\|_X \lesssim \Big( \sum_{j} \|x_j\|_X^p \Big)^{1/p}
    \end{equation*}
    holds, with implicit constant independent of the sequence.
    On the other hand, $X$ has \emph{cotype $q$} if the reverse estimate holds with $p$ replaced by $q$.
  \item The space $X$ has \emph{martingale type $p$} if for all finite $X$-valued $L^p$-martingales $(f_j)_{j=0}^N$ (on any measure space $S$, with respect to any $\sigma$-finite filtration) the estimate
    \begin{equation*}
      \|f_N\|_{L^p(S;X)} \lesssim \Big( \|f_0\|_{L^p(S;X)}^p + \sum_{j=1}^N \|f_j - f_{j-1}\|_{L^p(S;X)}^p \Big)^{1/p}
    \end{equation*}
    holds with implicit constant independent of the martingale.
    On the other hand, $X$ has \emph{martingale cotype $q$} if the reverse estimate holds with $p$ replaced by $q$.
  \end{itemize}
\end{defn}

If $X$ has martingale type $p$, then it follows that $X$ has type $p$, but in general the converse is false.
However, for UMD spaces, type $p$ is equivalent to martingale type $p$.
The same statements are true for (martingale) cotype.
See for example \cite[\textsection 10.4]{gP16}.

Martingale cotype $q$ can be characterised by the vector-valued L\'epingle inequality on the variation of martingale difference sequences, as shown by Pisier and Xu \cite{PX88}.
Thus martingale cotype is tightly linked with variational estimates.
We have in particular the following result for convolutions, which follows immediately from \cite[Lemma 3.3]{HLM19}.\footnote{This result follows from the vector-valued variational inequalities for averaging operators proved in \cite{HM17}, via an argument from \cite{CJRW00}}

\begin{thm}\label{thm:cotype-variational-convolution}
  Let $X$ be a Banach space with martingale cotype $r_0 \in [2,\infty)$, and let $\phi \in \Sch(\R)$.
  Then for all $r > r_0$ we have
  \begin{equation*}
    \big\| \|t \mapsto (f \ast \Dil_t \phi)(\cdot) \|_{V^r(\R_+;X)} \big\|_{L^p(\R)} \lesssim_{p,\phi,X} \|f\|_{L^p(\R;X)}
  \end{equation*}
  for all $p \in (1,\infty)$.
  In particular this estimate holds if $X$ is UMD with cotype $r_0$.
\end{thm} 

Every Banach space $X$ has type $1$ and cotype $\infty$; the situation is more interesting if $X$ has nontrivial type (type $p$ for some $p > 1$) or finite cotype (cotype $q$ for some $q < \infty$).
Here's what we need to know about type and cotype.
\begin{itemize}
\item UMD spaces have nontrivial type and finite cotype.
\item $r$-intermediate UMD spaces have type $r'$ and cotype $r$.
\item Type dualises: if $X$ has type $p$, then $X^*$ has cotype $p'$.
\item Cotype dualises in presence of nontrivial type: if $X$ has cotype $q$ \emph{and nontrivial type}, then $X^*$ has type $q'$.\footnote{$\ell^1$ has cotype $2$, but does not have nontrivial type. Its dual, $\ell^\infty$, has neither nontrivial type nor finite cotype.}
\item If $X$ has type $p$ and cotype $q$, then for each Hilbert space $H$ the space of $\gamma$-radonifying operators $\gamma(H,X)$ (defined very soon) also has type $p$ and cotype $q$.
\end{itemize}

\subsubsection{$\gamma$-radonifying operators}\label{sec:gamma}

Littlewood--Paley estimates are always phrased in terms of square functions.
When dealing with UMD-valued functions, the appropriate notion of `square function' is not an $L^2$-norm, but rather a \emph{$\gamma$-radonifying norm}, as defined below.
This concept is described in more depth in \cite[\textsection 2]{AU19-2} and \cite[Chapter 9]{HNVW17}.

\begin{defn}
  Let $H$ be a Hilbert space and $X$ a Banach space.
  A linear operator $\map{T}{H}{X}$ is \emph{$\gamma$-summing} if
  \begin{equation*}
    \|T\|_{\gamma_\infty(H,X)} := \Big( \sup \E \Big\| \sum_j \gamma_j Th_j \Big\|_X^2 \Big)^{1/2} < \infty
  \end{equation*}
  with supremum taken over all finite orthonormal systems $(h_j)$ in $H$, with $(\gamma_j)$ a sequence of independent standard Gaussian random variables on some probability space.
  This norm defines the Banach space of $\gamma$-summing operators $\gamma_\infty(H,X)$.
  All finite rank operators $H \to X$ are $\gamma$-summing, and the Banach space of \emph{$\gamma$-radonifying operators} $\gamma(H,X)$ is the closure of the finite rank operators in $\gamma_\infty(H,X)$.
\end{defn}

\begin{rmk}
  When $X$ has finite cotype, and in particular if $X$ is UMD, we have $\gamma_\infty(H,X) = \gamma(H,X)$ for all Hilbert spaces $H$.
\end{rmk}

The space $\gamma(H,X)$ is a space of operators, but when $H$ is taken to be an $L^2$ space, it contains many elements that can be identified with $X$-valued functions.

\begin{defn}
  For a Banach space $X$ and a measure space $(S,\mc{A},\mu)$ we define
  \begin{equation*}
    \gamma(S;X) = \gamma_\mu(S;X) := \gamma(L_{\mu}^2(S),X).
  \end{equation*}
  For a strongly measurable, weakly $L^2$ function $\map{f}{S}{X}$, we write $f \in \gamma(S;X)$ to mean that the Pettis integral operator $\map{\II_f}{L^2(S)}{X}$ defined by
  \begin{equation*}
    \II_f g := \int_S f(s) g(s) \, \dd\mu(s)
  \end{equation*}
  is in $\gamma(S,X)$, and write
  \begin{equation*}
    \|f\|_{\gamma(S;X)} := \|\II_f\|_{\gamma(S;X)}.
  \end{equation*}
\end{defn}

When $H$ is a Hilbert space and $S$ a measure space, we have $\gamma(S;H) = L^2(S;H)$ with equivalent norms \cite[Proposition 9.2.9]{HNVW17}.
In general one should think of $\gamma(S;X)$ as a function space analogous to $L^2(S;X)$, but better adapted to the geometry of $X$.
One manifestation of this analogy is the $\gamma$-H\"older inequality \cite[Theorem 9.2.14]{HNVW17}.

\begin{prop}\label{prop:gamma-duality-holder}
  Let $(S,\mc{A},\mu)$ be a measure space and $X$ a Banach space.
  Suppose $\map{f}{S}{X}$ and $\map{g}{S}{X^*}$ are in $\gamma(S;X)$ and $\gamma(S;X^*)$ respectively.
  Then $\langle f; g\rangle \to \C$ is integrable, with
  \begin{equation*}
    \int_S |\langle f; g\rangle| \, \dd\mu \leq \|f\|_{\gamma(S;X)} \|g\|_{\gamma(S;X^*)}.
  \end{equation*}
  Conversely, if $X$ is UMD and $\map{f}{S}{X}$ is strongly measurable and weakly $L^2$, then $f \in \gamma(S;X)$ if and only if there exists a constant $C < \infty$ such that
  \begin{equation*}
    \Big| \int_S \langle f;g \rangle \, \dd\mu \Big| \leq C\|g\|_{\gamma(S;X^*)}
  \end{equation*}
  for all $g \in L^2(S) \otimes X^*$, in which case $\|f\|_{\gamma(S;X)} \lesssim_X C$.
\end{prop}

\begin{rmk}\label{rmk:K-cvx}
  The converse statement in Proposition \ref{prop:gamma-duality-holder} holds when $X$ is \emph{$K$-convex}, a property implied by UMD that we will not discuss here.
\end{rmk}

The $\gamma$-spaces satisfy a Fubini-type relation with $L^p$-spaces, which shows one way in which $\gamma$-norms are \emph{not} like classical $L^2$-norms.
See \cite[Theorem 9.4.8]{HNVW17}.

\begin{thm}\label{thm:gamma-fubini}
  Let $(S,\mc{A},\mu)$ be a $\sigma$-finite measure space, let $H$ be a Hilbert space, and $p \in [1,\infty)$.
  Then the mapping $\map{U}{L^p(S;\gamma(H,X))}{\Lin(H,L^p(S;X))}$ defined by
  \begin{equation*}
    (Uf)h := f(\cdot)h \qquad \forall h \in H
  \end{equation*}
  defines an isomorphism
  \begin{equation*}
    L^p(S;\gamma(H;X)) \cong \gamma(H;L^p(S;X)).
  \end{equation*}
\end{thm}

When $S_1$ and $S_2$ are $\sigma$-finite measure spaces and $\map{f}{S_1 \times S_2}{X}$ is sufficiently nice, this says in particular that
\begin{equation*}
  \|f\|_{L^p(S_1;\gamma(S_2;X))} \simeq \|f\|_{\gamma(S_2; L^p(S_1;X))}.
\end{equation*}

UMD Banach spaces allow for various forms of Littlewood--Paley theory.
We need the following estimate, written in terms of $\gamma$-norms, which we proved in \cite[Theorem 2.18]{AU19-2}.

\begin{prop}
  \label{prop:UMD-littewood-paley}
  Let $X$ be a UMD Banach space and $p \in (1,\infty)$.
  Fix a mean zero Schwartz function $\psi \in \Sch(\R)$.
  Then for all $f \in L^p(\R;X)$,
  \begin{equation*}
    \| f \ast \Dil_t \psi\|_{L^p (\R; \gamma_{\dd t/t}(\R_+; X))} \lesssim_{X,p,\psi} \|f\|_{L^p(\R;X)}.
  \end{equation*}
  The implicit constant depends only on the second-order Schwartz seminorms of $\psi$ and the distance of $\spt \hat{\psi}$ to the origin.
\end{prop}

We will make use of the following lemma, which controls the Bochner integral of the product of a vector-valued function and a scalar-valued function by the product of corresponding $L^2$- and $\gamma$-norms.

\begin{lem}
  \label{lem:gamma-L1-emb}
  Let $X$ be a Banach space and $(S,\mc{A},\mu)$ a measure space.
  Then for all separably valued $\map{f}{S}{X}$ in $\gamma(S;X)$ and all $g \in L^2(S;\CC)$, the function $\map{gf}{S}{X}$ is strongly integrable and
  \begin{equation*}
    \Big\| \int_S g(s)f(s) \, \dd\mu(x) \Big\|_X \leq  \|g\|_{L^2(S)} \|f\|_{\gamma(S;X)}.
  \end{equation*}
\end{lem}

\begin{proof}
  Since $f$ is weakly $L^2$ and $g \in L^2(S)$, we find that $gf$ is weakly integrable.
  By the Pettis measurability theorem, since $f$ is separably valued, $gf$ is strongly integrable.
  Duality and Proposition \ref{prop:gamma-duality-holder} then yield
  \begin{equation*}
    \begin{aligned}
      \Big\| \int_S g(s)f(s) \, \dd\mu(s) \Big\|_X
      &= \sup_{\substack{x^* \in X \\ \|x^*\|_{X^*} = 1}} \Big| \Big\langle \int_S g(s)f(s) \, \dd\mu(s) ; x^* \Big\rangle \Big| \\
      &\leq \sup_{x^*} \int_S |\langle f(s); \overline{g(s)}x^* \rangle| \, \dd\mu(s) \\
      &\leq \sup_{x^*} \|f\|_{\gamma(S;X)} \|\overline{g} \otimes x^*\|_{\gamma(S;X^*)} \\
      &= \|f\|_{\gamma(S;X)} \|g\|_{L^2(S)}
    \end{aligned}
  \end{equation*}
  as claimed.
\end{proof}

\subsubsection{$R$-bounds}

No discussion of Banach-valued analysis is complete without a section on the $R$-bound of a family of operators.
This important concept is explained much more thoroughly in \cite[Chapter 8]{HNVW17}.

\begin{defn}
  Let $X$ and $Y$ be Banach spaces and $\mc{T} \subset \Lin(X,Y)$ a set of operators.
  The set $\mc{T}$ is \emph{$R$-bounded} if there exists a constant $C < \infty$ such that for all finite sequences $(T_j)$ in $\mc{T}$ and $(x_j)$ in $X$,
  \begin{equation*}
    \E \Big\| \sum_{j} \varepsilon_j T_j x_j \Big\|_Y \leq C \E \Big\| \sum_{j} \varepsilon_j x_j \Big\|_X.
  \end{equation*}
  The infimum of all possible $C$ in this estimate is called the $R$-bound of $\mc{T}$, and denoted by $R(\mc{T})$.
\end{defn}

$R$-boundedness is generally stronger than uniform boundedness (unless $Y$ has type $2$ and $X$ has cotype $2$, in which case the two notions are equivalent).
The $R$-bound is to the $\gamma$-norm as $L^\infty$ is to $L^2$; this is made precise in the following theorem (see \cite[Theorem 9.5.1 and Remark 9.5.8]{HNVW17}), and it is why we need to discuss $R$-bounds.

\begin{thm}[$\gamma$-multiplier theorem]
  Let $X$ and $Y$ be Banach spaces with finite cotype, and $(S,\mc{A},\mu)$ a measure space.
  Let $\map{A}{S}{\Lin(X,Y)}$ be such that for all $x \in X$ the $Y$-valued function $s \mapsto A(s)(x)$ is strongly $\mu$-measurable, and that the range $A(S)$ is $R$-bounded.
  Then for every $\map{f}{S}{X}$ in $\gamma(S;X)$, the function $\map{Af}{S}{Y}$ is in $\gamma(S;Y)$, with
  \begin{equation*}
    \|Af\|_{\gamma(S;Y)} \lesssim R(A(S)) \|f\|_{\gamma(S;X)}.
  \end{equation*}
\end{thm}

In particular, for a Banach space $X$ with finite cotype we have 
\begin{equation*}
  \|a \cdot f\|_{\gamma(S;X)} \lesssim \|f\|_{\gamma(S;X)}
\end{equation*}
for all $a \in L^\infty(S)$, as each bounded set of scalar operators $\{cI : c \in \C, |c| \leq N\}$ is $R$-bounded as a consequence of Kahane's contraction principle for Rademacher sums.

Controlling the $R$-bound of a family of operators is generally subtle, and depends strongly on the structure of the family.
We will need to consider families of convolution operators in order to prove an important technical lemma (Lemma \ref{lem:short-variation}).
But first we'll need $R$-bounds for families of Fourier multipliers with symbols of bounded variation, which is a special case of \cite[Theorem 8.3.4]{HNVW17}.
For a bounded function $m \in L^\infty(\R;\C)$, let $T_m$ denote the Fourier multiplier with symbol $m$.

\begin{thm}
  \label{thm:variation-Rbd}
  Let $X$ be a UMD space and $p \in (1,\infty)$.
  Then for all $C > 0$,
  \begin{equation*}
    R\big(\{T_m \in \Lin(L^p(\R;X)) : \|m\|_{V^1(\R;\C)} \leq C\}\big) \lesssim_{p,X} C.
  \end{equation*}
\end{thm}

As a corollary we obtain $R$-bounds for families of convolutions with dilates of a fixed nice (but not necessarily Schwartz) function. 
Later this will be applied to the functions $\phi(x) = \langle x \rangle^{-N}$ for $N$ sufficiently large.

\begin{cor}
  \label{cor:R-bd-bump}
  Let $X$ be a UMD space and $p \in (1,\infty)$.
  Fix an integrable function $\phi \in L^1(\R)$ such that $\hat{\phi}$ has bounded variation. 
  Then 
  \begin{equation*}
    R\big(\{[f \mapsto f \ast \Dil_t \phi] \in \Lin(L^p(\R;X)) : t > 0\} \lesssim_{p,X} \|\hat{\phi}\|_{V^1(\R;\C)}.
  \end{equation*}
\end{cor}

\begin{proof}
  For each $t > 0$ we have $\|\widehat{\Dil_t \phi}\|_{V^1(\RR;C)} = \|\hat{\phi}\|_{V^1(\RR;C)}$ since $\widehat{\Dil_t \phi}$ is a reparametrisation of $\hat{\phi}$.
  The result then follows from Theorem \ref{thm:variation-Rbd}, since $f \mapsto f \ast \Dil_t \phi$ is the Fourier multiplier with symbol $\widehat{\Dil_t \phi}$.
\end{proof}

This result extends to families of generalised convolutions, in which the convolving function (`$g$' in `$f \ast g$') is allowed to depend on the variable $x \in \R$.
In what follows it will be convenient to write
\begin{equation*}
  \Conv_g f := f \ast g
  \quad \text{and} \quad 
  \Mult_g f := fg.
\end{equation*}

\begin{rmk}\label{rmk:pws}
  In the following two lemmas we will speak of \emph{piecewise smooth} functions $\map{P}{\R \times \R_+}{\Sch(\R)}$.
  By this we mean that there exists a sequence of points $0 = t_0 < t_1 < t_2 < \cdots$ with $t_n \to \infty$ such that the restrictions $P|_{\R \times (t_i, t_{i+1})}$ are smooth.
  We use this concept purely to avoid discussion of measurability of $\Sch(\R)$-valued functions.
\end{rmk}

\begin{lem}
  \label{lem:moving-convolutions}
  Let $X$ be a UMD Banach space and $p \in (1,\infty)$.
  Given a piecewise smooth function $\map{P}{\RR \times \RR_+}{\Sch(\RR)}$, for each $t > 0$ define the bounded operator $\map{C_{P,t}}{L^p(\RR;X)}{L^p(\RR;X)}$ by
  \begin{equation*}
    (C_{P,t} f)(x) := (f \ast \Dil_t P_{x,t})(x).
  \end{equation*}
  Then we have the $R$-bound
  \begin{equation*}
    R\{ C_{P,t} : t > 0 \} \lesssim_{p,X} \sup_{x,t} \|P_{x,t}\|_{*}
  \end{equation*}
  where $\|\cdot\|_*$ is a Schwartz seminorm of sufficiently high order.
\end{lem}

\begin{proof}
  First assume $f \in \Sch(\RR;X)$.
  Fix $N$ large, and write by Fourier inversion
  \begin{equation*}
    \begin{aligned}
      (f \ast \Dil_t P_{x,t})(x)
      &= \int_\R f(y) \Dil_t P_{x,t}(x-y) \, \dd y \\
      &= \int_\R f(y) \frac{1}{t} \Big\langle  \frac{x-y}{t} \Big\rangle^{-N} \tilde{P}_{x,t}\Big(\frac{x-y}{t}\Big) \, \dd y \\
      &= \int_\R f(y) \frac{1}{t} \Big\langle \frac{x-y}{t} \Big\rangle^{-N} \int_\R \widehat{\tilde{P}_{x,t}}(\xi) e^{i\xi(x-y)/t} \, \dd\xi \, \dd y \\
      &= \int_\RR \int_\RR e^{-i\xi y/t} f(y) \frac{1}{t}  \Big\langle \frac{x-y}{t} \Big\rangle^{-N} \, \dd y \, e^{i\xi x/t} \widehat{\tilde{P}_{x,t}}(\xi) \, \dd\xi \\
      &= \int_\RR (\Mod_{\xi/t} f \ast \Dil_t J_N)(x) P_{\xi,t}^\heartsuit(x) \, \langle \xi \rangle^{-N} \dd\xi \\
    \end{aligned}
  \end{equation*}
  where
  \begin{equation*}
    \begin{aligned}
      \tilde{P}_{x,t}(w) &:= \langle w \rangle^N P_{x,t}(w) \\
      J_{N}(w) &:= \langle w \rangle^{-N} \\
      P^\heartsuit_{\xi,t}(x) &:= \langle \xi \rangle^N \Mod_{-x/t} \widehat{\tilde{P}_{x,t}}(\xi).
    \end{aligned}
  \end{equation*}
  Thus we have
  \begin{equation*}
    (f \ast \Dil_t P_{x,t})(x) = \int_\R (T_{\xi,t} f)(x)   \, \langle \xi \rangle^{-N} \dd \xi
  \end{equation*}
  for all $x \in \R$,
   where
   \begin{equation}
     \label{eq:TP-op-defn}
    T^P_{\xi,t} := \Mult_{P_{\xi,t}^\heartsuit} \Conv_{\Dil_t J_N} \Mod_{\xi/t} .
  \end{equation}
  By density of the Schwartz functions in $L^p(\R;X)$, we get the representation
  \begin{equation*}
    C_{P,t} = \int_\R T^P_{\xi,t} \, \langle \xi \rangle^{-N} \dd\xi
  \end{equation*}
  in the strong operator topology on $L^p(\R;X)$.
  Using \cite[Theorem 8.5.2]{HNVW17}, it follows that
  \begin{equation}
    \label{eq:R-bd-reduction}
    R\{ C_{P,t} : t > 0\} 
    \lesssim R\{  T_{\xi,t}^P : t > 0, \xi \in \R \}.
  \end{equation}
  Since the functions $P_{\xi,t}^\heartsuit$ satisfy
  \begin{equation*}
    \sup_{\xi \in \R, t > 0} \|P_{\xi,t}^\heartsuit\|_\infty
    \lesssim \sup_{x \in \R, t > 0} \sup_{\xi \in \R} | \langle \xi \rangle^N  \widehat{\tilde{P}_{x,t}}(\xi) |
    \leq \sup_{x \in \R, t > 0} \|P_{x,t}\|_*
  \end{equation*}
  the operators $\{\Mult_{P_{\xi,t}^\heartsuit} : \xi \in \R\}$ have $R$-bound controlled by $\sup_{x , t > 0} \|P_{x,t}\|_*$.
  Furthermore, by Corollary \ref{cor:R-bd-bump} the operators $\{\Conv_{\Dil_t J_N} : t > 0\}$ have $R$-bound controlled by a constant depending on $p$ and $X$ (for sufficiently large $N$).
  Finally, since modulations are just multiplication by unimodular functions, the operators $\{\Mod_{\xi/t} : \xi \in \R\}$ have $R$-bound controlled by $1$.
  Using \eqref{eq:R-bd-reduction} and \eqref{eq:TP-op-defn}, the claim follows. 
\end{proof}

\subsubsection{An important technical lemma}

All of the concepts described above---UMD, cotype, $\gamma$-norms, $R$-bounds---come together in a key technical lemma that we will use in bounding various error terms.

\begin{lem}
  \label{lem:short-variation}
  Let $Y$ be a UMD space with cotype $r \geq 2$.
  Let $N_\pm \colon \N \times \RR \to \RR_+$ be measurable functions.
  For each $z \in \R$ let $I_{j,z}$ be the interval $(N_-(j,z),N_+(j,z))$, and suppose that the intervals $(I_{j,z})_{j \in \N}$ have finite overlap.
  Fix a measurable function $\map{q}{\N \times \RR \times \RR_+}{\CC}$ and a function $\map{P}{\N \times \RR \times \RR_+}{\Sch(\RR)}$ such that for each $j \in \N$, the function $\map{P_j}{\R \times \R_+}{\Sch(\R)}$ is piecewise smooth (in the sense of Remark \ref{rmk:pws}).
  Then for all $h \in L^2(\RR; \gamma(\RR;Y))$ we have
  \begin{equation*}
    \begin{aligned}
      &\Big\|  \int_{I_{j,z}} h(\cdot,\sigma) \ast \Dil_\sigma P_{j,z,\sigma} (z) q_{j,z}(\sigma) \, \frac{\dd\sigma}{\sigma} \Big\|_{L_{\dd z}^2(\RR; \ell_j^r(\N;Y))} \\
      &\qquad \lesssim \|h\|_{L^2(\RR; \gamma(\RR;Y))} \sup_{j,z,\sigma} \|P_{j,z,\sigma}\|_* \sup_{j,z} \|q_{j,z}(\sigma)\|_{L^2_{\dd\sigma/\sigma}(I_{j,z})}
    \end{aligned}
  \end{equation*}
  where $\|\cdot\|_*$ denotes a Schwartz seminorm of sufficiently high order.
\end{lem}

\begin{proof}
  For $j \in \N$ and $z \in \R$ we have
  \begin{equation*}
    \begin{aligned}
      &\Big\|\int_{I_{j,z}} h(\cdot,\sigma) \ast \Dil_\sigma P_{j,z,\sigma} (z) q_{j,z}(\sigma) \frac{\dd\sigma}{\sigma} \Big\|_{Y} \\
      &\leq \|q_{j,z}\|_{L^2_{\dd\sigma/\sigma}(I_{j,z})} \| \1_{I_{j,z}}(\sigma) h(\cdot,\sigma) \ast \Dil_\sigma P_{j,z,\sigma} (z)\|_{\gamma_{\dd\sigma/\sigma}(\RR_+;Y)}
    \end{aligned}
  \end{equation*}
  by Lemma \ref{lem:gamma-L1-emb}.
  Since $\gamma_{\dd\sigma/\sigma}(\RR_+;Y)$ has cotype $r$ we then get
  \begin{equation*}
    \begin{aligned}
      &\Big\|\int_{I_{j,z}} h(\cdot,\sigma) \ast \Dil_\sigma P_{j,z,\sigma} (z) q_{j,z}(\sigma) \frac{\dd\sigma}{\sigma} \Big\|_{\ell_j^r(\N;Y)} \\
      &\leq \sup_{j \in \N} \|q_{j,z}\|_{L^2_{\dd\sigma/\sigma}(I_{j,z})} \Big( \sum_{j \in \N} \| \1_{I_{j,z}}(\sigma) h(\cdot,\sigma) \ast \Dil_\sigma P_{j,z,\sigma} (z)\|_{\gamma_{\dd\sigma/\sigma}(\RR_+;Y)}^r \Big)^{1/r} \\
      &\lesssim \sup_{j \in \N} \|q_{j,z}\|_{L^2_{\dd\sigma/\sigma}(I_{j,z})} \E \Big\| \sum_{j \in \N} \varepsilon_j \1_{I_{j,z}}(\sigma) h(\cdot,\sigma) \ast \Dil_\sigma P_{j,z,\sigma} (z)  \Big\|_{\gamma_{\dd\sigma/\sigma}(\RR_+; Y)} \\
      &\lesssim \sup_{j \in \N} \|q_{j,z}\|_{L^2_{\dd\sigma/\sigma}(I_{j,z})} \Big\| \sum_{j \in \N} \1_{I_{j,z}}(\sigma) h(\cdot,\sigma) \ast \Dil_\sigma P_{j,z,\sigma} (z)  \Big\|_{\gamma_{\dd\sigma/\sigma}(\RR_+; Y)}
    \end{aligned}
  \end{equation*}
  using finite overlap of the intervals $(I_{j,z})_{j \in \N}$.
  Taking the $L^2$-norm in $z$ and applying the $\gamma$-Fubini theorem gives
  \begin{equation*}
    \begin{aligned}
      &\Big\|  \int_{I_{j,z}} h(\cdot,\sigma) \ast \Dil_\sigma P_{j,z,\sigma} (z) q_{j,z}(\sigma) \frac{\dd\sigma}{\sigma} \Big\|_{L_{\dd z}^2(\RR; \ell_j^r(\N;Y))} \\
      &\lesssim \sup_{j,z} \|q_{j,z}\|_{L^2_{\dd\sigma/\sigma}(I_{j,z})} \Big\| \sum_{j \in \N} \1_{I_{j,z}}(\sigma) h(\cdot,\sigma) \ast \Dil_\sigma P_{j,z,\sigma} (z)  \Big\|_{\gamma_{\dd\sigma/\sigma}(\RR_+; L^2_{\dd z}(\RR;Y))}.
    \end{aligned}
  \end{equation*}
  It remains to control the second factor.
  By the $\gamma$-multiplier theorem and another application of the $\gamma$-Fubini theorem, we have
  \begin{equation*}
    \begin{aligned}
      &\Big\| \sum_{j \in \N} \1_{I_{j,z}}(\sigma) h(\cdot,\sigma) \ast \Dil_\sigma P_{j,z,\sigma} (z)  \Big\|_{\gamma_{\dd\sigma/\sigma}(\RR_+; L^2_{\dd z}(\RR;Y))} \\
      &\qquad \lesssim R\{ C_{\tilde{P},\sigma} : \sigma > 0 \} \|h\|_{L^2_{\dd z}(\R; \gamma_{\dd\sigma/\sigma}(\R_+;Y))}
    \end{aligned}
  \end{equation*}
  where the operators $C_{\tilde{P},\sigma} \in \Lin(L^2_{\dd z}(\R;Y))$ are defined as in Lemma \ref{lem:moving-convolutions} with
  \begin{equation*}
    \tilde{P}_{z,\sigma} := \sum_{j \in \N} \1_{I_{j,z}}(\sigma) P_{j,z,\sigma}(z). 
  \end{equation*}
  Lemma \ref{lem:moving-convolutions} and finite overlap then gives
  \begin{equation*}
    R\{ C_{\tilde{P},\sigma} : \sigma > 0 \} \lesssim \sup_{z,\sigma} \|\tilde{P}_{z,\sigma}\|_*  \lesssim \sup_{j,z,\sigma} \|P_{j,z,\sigma}\|_*,
  \end{equation*}
  which completes the proof.
\end{proof}

%
  


\subsection{Outer Lebesgue spaces}\label{sec:OLS}

We quickly describe abstract outer spaces and the associated outer Lebesgue quasinorms.
We use the formulation in \cite{AU19-2}, to which the reader is referred for further details. 
For a topological space $\OX$ we let $\Bor(\OX)$ denote the $\sigma$-algebra of Borel sets in $\OX$, and for a Banach space $X$ we let $\Bor(\OX;X)$ denote the set of strongly Borel measurable functions $\OX \to X$.

\begin{defn}
  Let $\OX$ be a topological space.
  \begin{itemize}
  \item
    A \emph{$\sigma$-generating collection on $\OX$} is a subset $\OB \subset\Bor(\OX)$ such that $\OX$ can be written as a union of countably many elements of $\OB$.
    We write
    \begin{equation*}
      \OB^\cup := \bigg\{\bigcup_{n=1}^\infty B_n : \text{$B_n \in \OB$ for all $n \in \N$}\bigg\}.
    \end{equation*}
  \item
    A \emph{local measure} (on $\OB$) is a $\sigma$-subadditive function $\mu \colon\OB \to [0,\infty]$ such that $\mu(\emptyset)=0$.

  \item Given a Banach space $X$, an $X$-valued \emph{local size} (on $\OB$) is a collection of `quasi-norms' $\OS = (\| \cdot \|_{\OS(B)})_{B \in \OB}$, with each $\map{\|\cdot\|_{\OS(B)}}{\Bor(\OX;X)}{[0,\infty]}$, satisfying
  \begin{description}
  \item[positive homogeneity] $\| \lambda F \|_{\OS(B)}=|\lambda|\| F \|_{\OS(B)}$ for all $F \in \Bor(\OX;X)$ and $\lambda\in \C$;
  \item[global positive-definiteness] $\|F\|_{\OS(B)}=0$ for all $B\in\OB$ if and only if $F =0$;
  \item[quasi-triangle inequality] there exists a constant $C \in [1,\infty)$  such that $\| F+G \|_{\OS(B)}\leq C ( \| F \|_{\OS(B)}+\| G \|_{\OS(B)} )$ for all $B\in\OB$ and $F,G \in \Bor(\OX;X)$.
  \end{description}

\item An ($X$-valued) \emph{outer space} is a tuple $(\OX,\OB, \mu, \OS)$ consisting of a topological space $\OX$, a $\sigma$-generating collection $\OB$ on $\OX$, a local measure $\mu$, and an $X$-valued local size $\OS$, all as above. We often do not make reference to the Banach space $X$.

\end{itemize}
\end{defn}

Given an outer space $(\OX,\OB, \mu, \OS)$, we extend $\mu$ to an outer measure on $\OX$ via countable covers: for all $E \subset \OX$,
\begin{equation*}
  \mu(E):=\inf \bigg\{  \sum_{n\in\N}\mu(B_{n}) : B_{n}\in\OB,\, \bigcup_{n\in\N} B_{n}\supset E\bigg\}.
\end{equation*}
We abuse notation and write $\mu$ for both the local measure and the corresponding outer measure.
We define the \emph{outer size} (or \emph{outer supremum}) of $F\in \Bor(\OX;X)$ by
\begin{equation*}
  \| F \|_{\OS} := \sup_{B\in\OB}\sup_{V\in\OB^{\cup}}\| \1_{\OX \setminus V}F \|_{\OS(B)},
\end{equation*}
and the \emph{outer super-level measure} of $F$ by
\begin{equation*}
  \mu ( \| F \|_{\OS}>\lambda ) := \inf \{\mu(V) : V\in\OB^{\cup},\, \| \1_{\OX\setminus V}F \|_{\OS}\leq \lambda \} \qquad \forall \lambda \geq 0.
\end{equation*}
Finally we define
\begin{equation*}
  \mu(\spt(F)):= \mu \bigl( \| F \|_{\OS}>0 \bigr).
\end{equation*}

\begin{defn}
  Let $(\OX,\OB,\mu,\OS)$ be an $X$-valued outer space. We define the \emph{outer Lebesgue quasinorms} of a function $F\in\Bor(\OX;X)$, and their weak variants, by setting
  \begin{align*}
    &  \|F\|_{L_{\mu}^{p} \OS} := \Big( \int_{0}^{\infty} \lambda^{p} \mu(\| F \|_{\OS} > \lambda) \, \frac{\dd \lambda}{\lambda} \Big)^{1/p}   &&\forall p\in(0,\infty), \\
    &  \|F\|_{L_{\sigma}^{p,\infty} \OS} := \sup_{\lambda > 0} \lambda \,\mu(\| F \|_{\OS} > \lambda)^{1/p} &&\forall p\in(0,\infty), \\
    &\|F\|_{L_{\mu}^{\infty} \OS} := \| F \|_{\OS}. 
  \end{align*}
  It is straightforward to check that these are quasinorms (modulo functions $F$ with $\mu(\spt (F)) = 0$). 
\end{defn}

\begin{defn}
  Let $(\OX, \OB, \mu, \OS)$ be an outer space.
  Let $\OB'$ be a $\sigma$-generating collection on $\OX$, and let $\nu$ be a local measure on $\OB'$.
  Then for all $q \in (0,\infty)$ define the \emph{iterated local size} $\sL_\mu^q \OS$ on $\OB'$ (which depends on $\nu$) by
  \begin{equation}
    \label{eq:iterated-size}
    \| F \|_{\sL^{q}_{\mu}\OS (B')} := \frac{1}{\nu(B')^{1/q} }\| \1_{B'} F \|_{L^{q}_{\mu} \OS} \qquad(B' \in \OB').
  \end{equation}
  We call $(\OX,\OB',\nu,\sL^{q}_{\mu} \OS)$ an \emph{iterated} outer space.
\end{defn}

We will use the following properties of outer Lebesgue quasinorms.
The first is a Radon--Nikodym-type domination result.
The proof is a straightforward modification of \cite[Lemma 2.2]{gU16} and \cite[Proposition 3.6]{DT15}.

\begin{prop}\label{prop:outer-RN}
  Let $(\OX,\OB,\mu,\OS)$ be an outer space such that the outer measure generated by $\mu$ is $\sigma$-finite. Let $\mf{m}$ be a positive Borel measure on $\OX$ such that
  \begin{equation*}
    \int_B  \|F(x)\|_X \, \dd \mf{m} (x) \lesssim \| F \|_{\OS(B)}\, \mu(B) \qquad \forall B \in \OB, \, \forall F \in \Bor(\OX;X)
  \end{equation*}
  and
  \begin{equation*}
    \mu(A) = 0 \, \Rightarrow \, \mf{m}(A) = 0 \qquad  \forall A \in \Bor(\OX). 
  \end{equation*}
  Then we have 
  \begin{equation*}
     \int_\OX  \|F(x)\|_X  \, \dd \mf{m}(x)  \lesssim \|F\|_{L^1_\mu \OS} \qquad \forall F \in L^1(\OX,\mf{m};X).
   \end{equation*}
\end{prop}

Next is an `outer H\"older inequality'.
For a proof in the trilinear setting see \cite[Proposition 4.3]{AU19}.

\begin{prop}\label{prop:outer-holder}
  Let $\OX$ be a topological space, $\OB$ a $\sigma$-generating collection on $\OX$, and $\mu$ a local measure on $\OB$.
  Let $X$ be a Banach space, $\OS$ an $X$-valued local size, $\OS^*$ an $X^*$-valued local size, and $\OS_\CC$ a $\C$-valued local size (all on $\OB$).
  Suppose these local sizes satisfy the H\"older relation
  \begin{equation}
    \label{eq:abstract-size-holder}
    \| \langle F, G \rangle \|_{\OS_\CC} \lesssim \| F \|_{\OS} \| G \|_{\OS^*}
  \end{equation}
  for all $F \in \Bor(\OX;X)$ and $G \in \Bor(\OX;X^*)$, where $\langle F;G \rangle(x) := \langle F(x); G(x) \rangle$ for $x \in \OX$.
  Then for all $p \in (0,\infty]$ we have 
    \begin{equation}
    \label{eq:abstract-Lp-holder}
    \| \langle F ; G \rangle \|_{L_{\mu}^{1} \OS_\C} \lesssim_{p}  \| F \|_{L_{\mu}^{p} \OS} \| G \|_{L_{\mu}^{p'} \OS^*}. 
  \end{equation}
\end{prop}

Outer Lebesgue quasinorms are typically estimated by interpolation.
For a proof of the following result see \cite[Proposition 3.5]{DT15}.

%

\begin{prop}\label{prop:outer-interpolation}
  Let $(\OX, \OB, \mu, \OS)$ be an $X$-valued outer space.
  Let $\Omega$ be a $\sigma$-finite measure space, and let $T$ be a quasi-sublinear operator mapping $L^{p_0}(\Omega;X) + L^{p_1}(\Omega;X)$ into $\Bor(\OX;X)$ for some $1 \leq p_0 < p_1 \leq \infty$.
  Suppose that
  \begin{equation*}
    \begin{aligned}
      \|Tf\|_{L_\mu^{p_0,\infty} \OS} &\lesssim \|f\|_{L^{p_0}(\Omega;X)}, \\
      \|Tf\|_{L_\mu^{p_1,\infty} \OS} &\lesssim \|f\|_{L^{p_1}(\Omega;X)}
    \end{aligned}
    \qquad \forall f \in L^{p_0}(\Omega;X) + L^{p_1}(\Omega;X).
  \end{equation*}
  Then for all $p \in (p_0,p_1)$,
  \begin{equation*}
    \|Tf\|_{L^p_\mu \OS} \lesssim \|f\|_{L^p(\Omega;X)} \qquad \forall f \in L^p(\Omega;X).
  \end{equation*}
\end{prop}

\subsection{The time-frequency-scale space}\label{sec:TFS}

We recall a number of notions from \cite{AU19-2}.
The time-frequency-scale space is $\RR^3_+$.
For $(\xi,x,s) \in \R^3_+$ we have mutually inverse \emph{local coordinate maps} $\pi_{(\xi,x,s)}$, $\pi_{(\xi,x,s)}^{-1}$ mapping $\RR^3_+$ to itself, defined by
\begin{equation}
  \label{eq:coordinate-maps}
  \begin{aligned}
    \pi_{(\xi,x,s)}(\theta,\zeta,\sigma) &:= (\xi + \theta(s\sigma)^{-1}, x + s\zeta, s\sigma) \\
    \pi_{(\xi,x,s)}^{-1}(\theta,\zeta,\sigma) &:= \Big( t(\xi - \eta), \frac{y-x}{s}, \frac{t}{s} \Big).
  \end{aligned}
\end{equation}

\begin{defn}\label{def:tree}
  Fix intervals $\Theta$ and $\Theta_{in}$, with $0 \in \Theta_{in} \subset \Theta$.
  Define the \emph{model tree} (over $\Theta$) by  
  \begin{equation}
    \label{eq:model-tree}
    \begin{aligned}
      &\mT_{\Theta}:=\big\{(\theta,\zeta,\sigma)\in \R^{3}_{+}\colon \theta \in \Theta,\,|\zeta|<1-\sigma \big\}
    \end{aligned}
  \end{equation}
  and the \emph{tree with top $(\xi,x,s) \in \R^{3}_{+}$}
  \begin{equation*}
    T_{(\xi,x,s)} := \pi_{(\xi,x,s)}(  \mT_\Theta)
  \end{equation*}
  (we supress $\Theta$ in the notation here, otherwise it gets in the way).
  For $T = T_{(\xi,x,s)}$ we write $\pi_T := \pi_{(\xi,x,s)}$ and $(\xi_T, x_T, s_T) = (\xi,x,s)$.
  The \emph{inner} and \emph{outer} parts of $T$ are
  \begin{align*}
    T^{in}&=\pi_{(\xi_{T},x_{T},s_{T})}\big( \mT_{\Theta} \cap \{\theta \in \Theta_{in}\} \big), \\
    T^{out}&=\pi_{(\xi_{T},x_{T},s_{T})}\big( \mT_{\Theta} \cap \{\theta \notin \Theta_{in}\} \big),
  \end{align*}
  and we denote the family of all trees by $\TT_\Theta$.
  We define a local measure $\mu$ on the $\sigma$-generating collection $\TT_\Theta$ by
  \begin{equation}\label{eq:TT-premeasure}
    \mu(T)  := s_T.
  \end{equation}
\end{defn}

Associated to each tree we have a pullback map (actually associated with the top of the tree) defined below.

\begin{defn}
  Let $X$ be a Banach space and $F \in \Bor(\R^3_+; X)$.
  For each $T \in \TT_\Theta$ define the function $\pi_T^* F \in \Bor(\R^3_+; X)$ by
  \begin{equation*}
    (\pi_{T}^{*}F)\,(\theta,\zeta,\sigma) := 
      \1_{\bar{\mT}}(\theta,\zeta,\sigma) \;e^{-2\pi i \xi_{T}(x_{T}+s_{T}\zeta )} F\circ\pi_{T}(\theta,\zeta,\sigma).
  \end{equation*}
\end{defn}

We also have \emph{strips}, which contain no frequency information.
These correspond to tents in $\R^2_+$.

\begin{defn}\label{def:strip}
  Define the \emph{model strip}  by 
  \begin{equation}
    \label{eq:model-strip}
    \begin{aligned}
      &\mD:=\big\{   (\zeta,\sigma)\in \R^{2}_{+}\colon |\zeta|<1-\sigma \big\},
    \end{aligned}
  \end{equation}
  and define the \emph{strip with top $(x,s) \in \R^{2}_+$}to be the set
  \begin{equation}\label{eq:strip}
    D_{(x,s)} := \pi_{(0,x,s)}(\RR \times \mD).
  \end{equation}
  We let $(x_D,s_D) := (x,s)$, and we denote the family of all strips by $\DD$.
  We define a local measure $\nu$ on the $\sigma$-generating collection $\DD$ by 
  \begin{equation}\label{eq:DD-premeasure}
    \nu(D_{(x,s)}):= s.
  \end{equation}
\end{defn}

Note that we have the expression
\begin{equation*}
  D_{(x,s)}=\bigl\{(\eta,y,t)\in \R^{3}_{+} \colon |y-x|<s-t\bigr\}=\bigcup_{\xi\in\Q} \pi_{(\xi,x,s)}(\mT_{\Theta})
\end{equation*}
for any interval $\Theta$ containing the origin, so each strip can be written as a countable union of trees in $\mT_{\Theta}$.

Now we define various local sizes, allowing us to measure Banach-valued functions on $\R^3_+$.
In the context of Theorem \ref{thm:main}, this Banach space will be either $X$ or $X^*$.
First we consider the following Banach-valued `Lebesgue' local sizes on $\TT_\Theta$.

\begin{defn}
  Let $\Theta$ and $\Theta_{in}$ be intervals with $0 \in \Theta_{in} \subset \Theta$, and let $Y$ be a Banach space.
  For $p \in [1,\infty]$ we define the local sizes $\LS_\Theta^{p}$ as follows: for $ F \in \Bor(\R^{3}_{+};Y)$ and $T \in \TT_\Theta$,
  \begin{equation}
    \label{eq:lL-size}
    \begin{aligned}
      \| F \|_{\LS_{\Theta}^{p}(T)}
      &:= \|  \pi^{*}_{T}F \|_{L_{\dd \theta \dd \zeta \frac{\dd \sigma}{\sigma}}^{p}(\mT_{\Theta};Y)} \\
      &=    \Big( \int_{\mT_{\Theta}} \big\| \pi_{T}^{*}F (\theta,\zeta,\sigma) \big\|_{Y}^{p} \, \dd \theta \, \dd \zeta \, \frac{\dd \sigma}{\sigma}\Big)^{\frac{1}{p}}
    \end{aligned}
  \end{equation}
  with the usual modification when $p=\infty$.
  These local sizes have `inner' and `outer' variants given by localisation to the inner and outer parts of trees: 
  \begin{equation*}
    \| F \|_{\LS_{\Theta, in}^{p}(T)} :=   \|\1_{T^{in}} F \|_{\LS_{\Theta}^{p}(T)}  \qquad   \| F \|_{\LS^{p}_{\Theta,out}(T)} :=   \|\1_{T^{out}} F \|_{\LS_{\Theta}^{p}(T)}.
  \end{equation*}
\end{defn}

We do not mention the target Banach space in the notation, as it is unambiguous from context.
We also omit mention of $\Theta_{in}$, otherwise the notation gets ridiculous. 
Generally multiple instances of $\LS$ can occur on the same line with respect to different Banach spaces, as in the case in the following local size-H\"older inequality.

\begin{prop}\label{prop:lebesgue-size-holder}
  Let $\Theta$ and $\Theta_{in}$ be intervals with $0 \in \Theta_{in} \subset \Theta$, and suppose $p,q \in (0,\infty]$.
  Let $Y$ be a Banach space, $F \in \Bor(\R^3_+; Y)$, and $G \in \Bor(\R^3_+; Y^*)$.
  Then for any $T \in \TT_{\Theta}$,
  \begin{equation*}
    \|\langle F; G \rangle\|_{\LS_{\Theta}^1(T)} \lesssim \|F\|_{(\LS_{\Theta,out}^{p} + \LS^{q}_{\Theta,in})(T)} \|G\|_{(\LS_{\Theta,out}^{p'} + \LS^{q'}_{\Theta,in})(T)}.
  \end{equation*}
\end{prop}

Next we define the `randomised' local sizes, using $\gamma$-norms (defined in Section \ref{sec:gamma}) to represent square functions.

\begin{defn}
  Let $\Theta$ and $\Theta_{in}$ be intervals with $0 \in \Theta_{in} \subset \Theta$, and let $Y$ be a Banach space.
  We define the local size $\RS_{\Theta}$ for $F \in \Bor(\R^3_+; Y)$ and $T \in \TT_{\Theta}$ by
  \begin{equation}\label{eq:R-lac-def}
    \|F\|_{\RS_{\Theta}(T)}
    :=  \Big( \int_{\R^{2}} \|\pi_T^*(\1_{T^{out}} F) (\theta,\zeta,\sigma) \|_{\gamma_{\dd \sigma / \sigma}(\R_{+}; Y)}^2 \, \dd \zeta \, \dd \theta  \Big)^{\frac{1}{2}}.
  \end{equation}
\end{defn}

Note that $\RS_{\Theta}$ depends on $\Theta$ and $\Theta_{in}$ through the definition of $T^{out}$.
We combine the randomised and Lebesgue local sizes into two `full local sizes' as follows.

\begin{defn}\label{defn:full-size}
  Let $\Theta$ and $\Theta_{in}$ be intervals with $0 \in \Theta_{in} \subset \Theta$, and let $Y$ be a Banach space.
  For $F \in \Bor(\R^{3}_{+} ; Y)$ and $T \in \TT_{\Theta}$ we define
  \begin{equation*}
    \| F \|_{\FS_{\Theta}(T)}:= \| F \|_{\RS_{\Theta}(T)} + \| F \|_{\LS_{\Theta}^{\infty}(T)}
  \end{equation*}
  and
  \begin{equation*}
    \| F \|_{\FS_{\Theta}^*(T)}:= \| F \|_{\RS_{\Theta}(T)} + \| F \|_{\LS^{1}_{\Theta,in}(T)}
  \end{equation*}
\end{defn}


The following local size-H\"older inequality follows simply from the definition of the local sizes, the classical H\"older inequality (for the Lebesgue parts), and the $\gamma$-H\"older inequality (for the randomised parts).

\begin{prop}
  \label{prop:size-Holder}
  Let $\Theta$ and $\Theta_{in}$ be intervals with $0 \in \Theta_{in} \subset \Theta$, and let $Y$ be a UMD Banach space.\footnote{The UMD property is not necessary. As noted in Remark \ref{rmk:K-cvx}, $K$-convexity would do.} 
  Let $F \in \Bor(\R^3_+; Y)$ and $G \in \Bor(\R^3_+ ; Y^*)$.
  Then for all $T \in \TT_{\Theta}$ and $A \in \TT_{\Theta}^\cup$,
  \begin{equation*}
    \| \1_{\R^3_+ \sm A} \langle F; G \rangle \|_{\LS_{\Theta}^1(T)} \lesssim \|F\|_{\FS_{\Theta}(T)} \|G\|_{\FS_{\Theta}^*(T)}.
  \end{equation*}
\end{prop}

By combining Proposition \ref{prop:size-Holder} with Proposition \ref{prop:outer-RN} we obtain a H\"older-type inequality involving the classical integral and iterated outer Lebesgue quasinorms.
See \cite[Corollary 4.13]{AU19} for full details of this argument in the Walsh setting.

\begin{cor}\label{cor:goal-outer-holder}\
  Let $\Theta$ and $\Theta_{in}$ be intervals with $0 \in \Theta_{in} \subset \Theta$, and let $Y$ be a UMD Banach space.
  Let $p, q \in (0,\infty]$.
  Then for all $F \in \Bor(\R^3_+; Y)$ and $G \in \Bor(\R^3_+ ; Y^*)$,
  \begin{equation*}
     \int_{\R^3_+} \big| \langle F(\eta,y,t) ; G(\eta,y,t) \rangle \big| \, \dd\eta \, \dd y \, \dd t \lesssim \|F\|_{L_{\nu}^{p} \sL_{\mu}^{q} \FS_{\Theta}} \|G\|_{L_{\nu}^{p'} \sL_{\mu}^{q'} \FS_{\Theta}^*} .
  \end{equation*}
\end{cor}

\subsection{(Truncated) wave packets and $\VCarl_{\mf{c}}$}

For each $(\eta,y,t) \in \R^3_+$ recall the operator
\begin{equation*}
  \Lambda_{(\eta,y,t)} := \Tr_{y}\Mod_{\eta}\Dil_{t}.
\end{equation*}
Given $\phi \in \Sch(\R)$, the function $\Lambda_{(\eta,y,t)}\phi$ is called a \emph{wave packet} at $(y,\eta,t)$ modelled on $\phi$.

In \cite{gU16} the second author introduced the concept of a \emph{family of left (or right) truncated wave packets}, defined below.
These are crucial to the representation of the operator $\VCarl_{\mf{c}}$ on the time-frequency-scale space $\R^3_+$.
The definition we use is a slight modification of that in \cite{gU16}, making the concept a bit clearer and easier to work with.
For $\mf{b} > 0$, we let $\Phi_{\mf{b}} \subset \Sch(\R)$ denote the subspace of Schwartz functions with Fourier transform supported in the ball $B_{\mf{b}}$.

\begin{defn}
  \label{defn:truncated-wave-packet}
  Let $\mf{b}, \epsilon > 0$.
  A family of \emph{left-truncated wave packets} with parameters $(\mf{b},\epsilon)$ is a smooth map
  \begin{equation*}
    \Psi^{+} \colon \R^{2}_{+} \times \{(c_{-},c_{+})\in(\R \cup \{\infty\})^{2} \colon c_{-}<c_{+}\} \to \Phi_{\mf{b}},
  \end{equation*}
  which we write as $\Psi^{(c_-,c_+),+}_{(\eta,t)} := \Psi^+((\eta,t),(c_-,c_+))$, satisfying the following properties:
  \begin{description}
  \item[Smoothness]
    For any fixed $N \in \N$ and $\alpha_{+},\alpha_{-},\alpha_{\eta},\alpha_{t}\in\{0,\dots,N\}$, the functions 
    \begin{equation*}
      \begin{aligned}
        & (t^{-1}\partial_{c_{-}})^{\alpha_{+}} (t^{-1}\partial_{c_{+}})^{\alpha_{-}}(t^{-1}\partial_{\eta})^{\alpha_{\eta}}(t\partial_{t})^{\alpha_{t}}\Psi_{(\eta,t)}^{(c_-,c_+),+} \in \Phi_{\mf{b}}
      \end{aligned}
    \end{equation*}
    are uniformly bounded in $\Phi_{\mf{b}}$;
    
  \item[Frequency dependence] the function $\Psi_{(\eta,t)}^{(c_-,c_+),+}$ is non-vanishing only if
    \begin{equation*}
      \eta\in \big( c_{-}+(1-\epsilon)t^{-1}, \min(c_+ - (1-\epsilon)t^{-1}, c_{-} + (1+\epsilon)t^{-1}) \big);
    \end{equation*}

    %
    
  \item[Weak dependence on right endpoint]
    For all $(\eta,t) \in \R_+^2$ and $c_- \in \R$, the map
    \begin{equation*}
      c_{+} \mapsto \Psi_{(\eta,t)}^{(c_-,c_+),+} 
    \end{equation*}
    is constant for $c_+ > \eta + 3t^{-1}$;

  \item[Frequency-scale relation for $(0,\infty)$]
    For all $(\eta,t) \in \R_+^2$,
    \begin{equation*}
      \Psi^{(0,\infty),+}_{(\eta, t)} = \Psi^{(0,\infty),+}_{(t\eta, 1)}.
    \end{equation*}
    
  \end{description}

  A family of \emph{right-truncated wave packets} with parameters $(\mf{b},\varepsilon)$ is a map
  \begin{equation*}
    \Psi^{-} \colon \R^{2}_{+}\times \{(c_{-},c_{+})\in(\R \cup \{-\infty\})^{2} \colon c_{-}<c_{+}\} \to \Phi_{\mf{b}}(\R)
  \end{equation*}
  such that the function  $(\eta,t,c_{-},c_{+}) \mapsto \Psi_{(\eta,t)}^{(-c_+, -c_-),-}$ is a family of left-truncated wave packets with parameters $(\mf{b},\varepsilon)$.
\end{defn}

\begin{rmk}
  If $\Psi^\pm$ is a family of left- or right-truncated wave packets with parameters $(\mf{b},\varepsilon)$ such that $\varepsilon$ is sufficiently small with respect to $\mf{b}$, it follows that  $\Mod_{\eta}\Dil_{t} \Psi_{(\eta,t)}^{(c_-,c_+),\pm}$ has Fourier support which is contained in $[c_{-},c_{+}]$ and Whitney with respect to $c_{\pm}$.
\end{rmk}

\begin{rmk}\label{rmk:WP-reduction}
  Consider a smooth map
  \begin{equation*}
    \Psi^{(0,1),+} \colon \R^2_+ \to \Phi_{\mf{b}}
  \end{equation*}
  satisfying the conditions above with $(c_-,c_+) = (0,1)$ fixed.
  Then families of left- and right-truncated wave packets $\Psi^\pm$: can be defined by rescaling:
  \begin{equation}\label{eq:wave-packet-symmetry}
    \Psi^{(c_{-},c_{+}),\pm}_{(\eta,t)}:=    \Psi^{(0,1),\pm}_{\big(\frac{\eta-c_{-}}{c_{+}-c_{-}},t(c_{+}-c_{-})\big)}.
  \end{equation}
  When $c_+ = \infty$ this has to be interpreted as a limit in $c_+ \to \infty$.
\end{rmk}

The representation of $\VCarl_{\mf{c}}$ on $\R^3_+$ follows from representing the Fourier projection onto an interval as a superposition of (modulated and dilated) truncated wave packets.
A version of following proposition was already shown in \cite[Lemma 3.1]{gU16}; here we give a somewhat clearer proof.

\begin{prop}\label{prop:multiplier-wave-packet-decomposition}
  For any sufficiently small $\mf{b}>0$ and any $\varepsilon > 0$ sufficiently small with respect to $\mf{b}$,  there exist families $ \Psi^{\pm}$ of left- and right-truncated wave packets with parameters $(\mf{b},\varepsilon)$ such that for any $c_{-}<c_{+}$ it holds that
  \begin{equation*}
    \1_{(c_{-},c_{+})}(\xi) = \sum_{\square \in \{+,-\}}  \iint_{\R^{2}_{+}}\big( \Lambda_{(\eta,0,t)} \Psi^{(c_{-},c_{+}),\square}_{(\eta,t)}\big)^{\wedge}(\xi) \, \dd \eta \, \dd t.
\end{equation*}
\end{prop}

\begin{proof}
  By Remark \ref{rmk:WP-reduction}, and since
  \begin{equation*}
    \1_{(c_{-},c_{+})}(\xi)= \1_{(0,1)}\Big(\frac{\xi-c_{-}}{c_{+}-c_{-}}\Big),
  \end{equation*}
  it is sufficient to prove the result for $(c_{-},c_{+})=(0,1)$.
  Let $\chi\in C^{\infty}_{c}(B_{\epsilon})$ be non-negative, even, and positive on $B_{\epsilon/2}$, set
\begin{equation*}
  \chi^{+}(z):=\|\chi\|_{L^{1}}^{-1}\int_{-\infty}^{z}\chi(\zeta)\dd \zeta\qquad  \chi^{-}(z):=\chi^{+}(-z),
\end{equation*}
and let $\FT{\phi} \in C^{\infty}_{c}(B_{\mf{b}/2})$ be non-negative, even, and positive on $B_{\mf{b}/4}$.

Consider the expression
\begin{equation}\label{eq:multiplier-plus}
m^{+}(\xi):=\int_{\R}\int_{0}^{\infty}\FT{\phi}(t(\xi-\eta))\chi(t\eta - 1)\chi^{-}(t(\eta-1) + 1) \, \dd \eta \, \dd t. 
\end{equation}
We claim that $m^{+}(\xi)$ is non-negative, smooth on $(0,1)$, supported in $[0,3/4]$, non-vanishing in an open neighborhood of $(0,1/2]$, and that $\partial_{\xi} m^{+}(\xi)$ is supported in $[1/4,3/4]$.

The first claim follows by noticing that $m^{+}(\xi)$ is an integral of non-negative functions. The integrand in \eqref{eq:multiplier-plus} vanishes unless $t\eta\in B_{\epsilon}(1)$, $t>2-2\epsilon$, and $t\xi\in B_{\mf{b}/2}(t\eta)$; thus the integral vanishes unless 
\begin{equation*}
  \begin{aligned}
    \xi \in &\bigcup_{t = 2-2\epsilon}^{\infty} B_{t^{-1}(\frac{\mf{b}}{2} + \varepsilon)}(t^{-1}) \\
    &=\bigcup_{t=2-2\epsilon}^{\infty} \Big( \frac{1-\epsilon - \mf{b}/2}{t}  , \frac{1+\epsilon + \mf{b}/2}{t}\Big) \subset (0,1).
  \end{aligned}
\end{equation*}
On the other hand, fix $t\eta=1$ and notice that for $t>2+\epsilon$ the integrand is non-vanishing for $t\xi\in B_{\mf{b}/4}(1)$.
Since the integrand depends continuously on the parameters $\eta,t$ we can see that
\begin{equation*}
  m^{+}(\xi)>0 \text{ on } \bigcup_{t=2+\epsilon}^{\infty} B_{\mf{b}/4t}(1/t) \supset (0,1/2]
\end{equation*}
as long as $\epsilon>0$ is sufficiently small with respect to $\mf{b}$. 
Furthermore for any $\xi>0$ we have that
\begin{equation*}
  \begin{aligned}
    \xi\partial_{\xi}m^{+}(\xi)
    &= \int_{\R}\int_{0}^{\infty} \xi\partial_{\xi}\FT{\phi}(t\xi-t\eta) \chi(t\eta-1) \chi^{-}(t(\eta-1)+1) \, \dd \eta \, \dd t
    \\
    &= \int_{\R}\int_{0}^{\infty}(t\partial_{t}-\eta \partial_{\eta})\FT{\phi}(t\xi-t\eta) \chi(t\eta-1) \chi^{-}(t(\eta -1)+1) \, \dd \eta \, \dd t
    \\
    & =-t\|\chi\|_{L^{1}}\int_{\R}\int_{0}^{\infty}\FT{\phi}(t\xi-t\eta) \chi(t\eta-1) \chi(t(\eta-1)+1) \, \dd \eta \, \dd t.
\end{aligned}
\end{equation*}
The integrand is non-vanishing only if $t\eta\in B_{\epsilon}(1)$, $t \in B_{2\epsilon}(2)$, and $t\xi\in B_{\epsilon+\mf{b}/2}(1)$ and in particular only if $\xi\in B_{\mf{b}/2}(1/2)$ (again, provided that $\epsilon$ is small enough with respect to $\mf{b}$).
We also deduce that $m^{+}(\xi)$ is smooth on $(0,1)$ since its derivative is given, by the discussion above, as an integral of smooth functions over a compact set of parameters: $t\eta\in B_{\epsilon}(1)$, $t \in B_{2\epsilon}(2)$. 

Now set 
\begin{equation*}
  \begin{aligned}
    &
    m^{-}(\xi)=m^{+}(1-\xi),
    \\
    &
    m(\xi)=
      \begin{cases}
        m^{-}(\xi)+m^{+}(\xi) &\xi\in(0,1)
        \\
        \lim_{\xi\to 0^{+}}m^{-}(\xi) & \xi<0
        \\
        \lim_{\xi\to 1^{-}}m^{-}(\xi) & \xi>1.
      \end{cases}
  \end{aligned}
\end{equation*}
From the discussion above it follows that $m(\xi)$ is positive and smooth, $\partial_{\xi}m(\xi)$ is supported in $B_{\mf{b}/2}(1/2)$, the identity $m(1-\xi)=m(\xi)$ holds, and
\begin{equation*}
  \begin{aligned}
  \1_{(0,1)} (\xi) &= \int_{\R^{2}_{+}}\FT{\phi}(t(\xi-\eta))m^{-1}(\xi)
  \left (
  \begin{aligned}
    &\chi(t\eta -1) \chi^{-}(t(\eta-1) + 1)\\
    &\quad+\chi^{+}(t\eta - 1) \chi(t(1-\eta) + 1)\\
\end{aligned}
  \right) \,
  \dd \eta  \, \dd t.
\end{aligned}
\end{equation*}
Notice that the integrand is non-vanishing only when $\eta\in B_{1}$ and $t>2-2\epsilon$, and for this range of parameters one has that $\FT{\phi}(\xi) m^{-1}(\xi/t+\eta)$ is contained in $\in C^{\infty}_{c}(B_{\mf{b}})$ and uniformly bounded in any Schwartz seminorm.
Now define
\begin{equation}
  \label{eq:truncated-wp}
  \begin{aligned}
    &
    \Psi_{(\eta,t)}^{(0,1),+}(x):=\chi(t\eta -1) \chi^{-}(t(\eta-1) + 1)\big( \FT{\phi}(\cdot)m^{-1}(\cdot/t+\eta)\big)^{\vee}(x) 
      \\
      &      \Psi_{(\eta,t)}^{(0,1),-}(x):=\chi^{+}(t\eta -1) \chi(t(\eta-1) + 1)\big( \FT{\phi}(\cdot)m^{-1}(\cdot/t+\eta)\big)^{\vee}(x).
    \end{aligned}
\end{equation}
The function $m^{-1}$ is constant on the support of $\xi\to \FT{\phi}(t(\xi-\eta))$ when
$t\eta<\frac{1-\mf{b}}{2}t -\frac{\mf{b}}{2}$ and in particular when 
\begin{equation*}
  \frac{1-\mf{b}}{2}t(\eta-1)<-\frac{1+\mf{b}}{2}t\eta -\frac{\mf{b}}{2}
\end{equation*}
or even when
\begin{equation*}
  t(\eta-1) < -(1+4\mf{b})<- \frac{1+2\mf{b}}{1-\mf{b}}(1+\epsilon)
\end{equation*}
as long as $\epsilon>0$ is small enough. Thus
\begin{equation*}
  \begin{aligned}
    &
    \Psi_{(\eta,t)}^{(0,1),+}(x)=\chi(t\eta -1) \chi^{-}(t(\eta-1) + 1)\;m^{-1}(0)\phi(x) 
      \\
      &      \Psi_{(\eta,t)}^{(0,1),-}(x)=\chi^{+}(t\eta -1) \chi(t(\eta-1) + 1)\;m^{-1}(0)\phi(x) 
    \end{aligned}
\end{equation*}
for  $ t(\eta-1) < -(1+4\mf{b})$.
It follows that $\Psi^{(0,1),\pm}$ satisfies the left- and right-truncated wave packet conditions.
\end{proof}

Now we can prove the claimed representation of $\VCarl_{\mf{c}}$.

\begin{cor}\label{cor:VCarl-WP}
  For $\mf{b} > 0$ sufficiently small, and $\varepsilon > 0$ sufficiently small with respect to $\mf{b}$,
  there exists $\phi\in\Phi_{\mf{b}}$ and families of left- and right-truncated wave packets $\Psi^{\pm}$ with parameters $(\mf{b}/2,\varepsilon)$ such that the wave packet representation
  \begin{equation}\label{eq:wave-packet-decomposition}
    \VCarl_{\mf{c},j} f(x)  
    = \sum_{\square \in \{+,-\}} \int_{\R^3_+} \langle f ; \Lambda_{(\eta,y,t)} \phi \rangle  \Lambda_{(\eta,y,t)}  \Psi_{(\eta,t)}^{(\mf{c}_{j}(x), \mf{c}_{j+1}(x)), \square}(x) \, \dd\eta \, \dd y \, \dd t
  \end{equation}
  holds for any measurable $\mf{c}\colon\R\to \Delta$ and $f\in\Sch(\R)$.
  We denote the summands on the right hand side by $\VCarl_{\mf{c},j}^\pm f(x)$.
\end{cor}

\begin{proof}
  Let $\phi\in\Phi_{\mf{b}}$ be even and real-valued with $\FT{\phi}=1$ on $B_{\mf{b}/2}$, and let $\Psi^{\pm}$ be families of truncated wave packets with parameters $(\mf{b}/2,\varepsilon)$ as in Proposition \ref{prop:multiplier-wave-packet-decomposition}.
  By Fourier support considerations we have
  \begin{equation*}
    \Lambda_{(\eta,0,t)} \Psi_{(\eta,t)}^{(c_{-},c_{+}),\pm} = \Lambda_{(\eta,0,t)}\phi * \Lambda_{(\eta,0,t)} \Psi_{(\eta,t)}^{(c_{-},c_{+}),\pm} 
  \end{equation*}
  for all frequencies $c_- < c_+$.
  Thus for all $x \in \R$ 
\begin{equation*}
  \begin{aligned}
    \VCarl_{\mf{c},j} f(x)
    &= \int_{\mf{c}_j(x)}^{\mf{c}_{j+1}(x)} \hat{f}(\xi) e^{2\pi i \xi x} \, \dd\xi \\
    &= \sum_{\square \in \{+,-\}}  \int_{\R^3_+}\big( \Lambda_{(\eta,0,t)} \Psi^{(\mf{c}_{j}(x),\mf{c}_{j+1}(x)),\square}_{(\eta,t)}\big)^{\wedge}(\xi)  \hat{f}(\xi)  e^{2\pi i \xi x} \, \dd\xi \, \dd \eta \, \dd t   \\
    &= \sum_{\square \in \{+,-\}} \int_{\R^2_+}
    f * \Lambda_{(\eta,0,t)} \phi * \Lambda_{(\eta,0,t)} \Psi_{(\eta,t)}^{(\mf{c}_{j}(x), \mf{c}_{j+1}(x)), \square}(x) \, \dd\eta \,  \dd t 
    \\
    &
    = \sum_{\square \in \{+,-\}}
    \int_{\R^3_+} f(z)  \Lambda_{(\eta,y,t)} \phi(z)  \Lambda_{(\eta,y,t)} \Psi_{(\eta,t)}^{(\mf{c}_{j}(x), \mf{c}_{j+1}(x)), \square}(x) \, \dd y \, \dd z\, \dd\eta \,  \dd t 
    \\
    &
    = \sum_{\square \in \{+,-\}}
    \int_{\R^3_+} \big\langle f; \Lambda_{(\eta,y,t)}\phi\big\rangle    \Lambda_{(\eta,y,t)} \Psi_{(\eta,t)}^{(\mf{c}_{j}(x), \mf{c}_{j+1}(x)), \square}(x) \, \dd\eta \, \dd y \, \dd t
  \end{aligned}
\end{equation*}
as claimed.
\end{proof}


\section{Reduction of the main theorem}
\label{sec:mainresult}

\subsection{Reduction to wave packet embedding bounds}

Now we can begin to approach Theorem \ref{thm:main-N}: for an $r_0$-intermediate UMD space $X$, we want to show
  \begin{equation*}
    \|\VCarl_{\mf{c}} f\|_{L^p(\R;l^{r}(\N;X))} \lesssim_{p,r,X} \|f\|_{L^p(\R;X)} \qquad \forall f \in \Sch(\R;X)
  \end{equation*}
  for all $r_0 < r < \infty$ and $(r/(r_0-1))' < p < \infty$, uniformly in $\map{\mf{c}}{\R}{\Delta}$.
  By duality, this is equivalent to showing
  \begin{equation}
    \label{eq:dual-bound}
    \Big|\int_{\R}\big\langle  \VCarl_{\mf{c}}f(x) ; g(x)\big \rangle \, \dd x\Big|\lesssim \|f\|_{L^{p}(\R;X)}\|g\|_{L^{p'}(\R;l^{r'}(\N;X^{*}))}
  \end{equation}
  for all $f\in\Sch(\R;X)$ and $g \in c_{00}(\N; \Sch(\R;X^*))$ (the set of finitely-supported sequences of $X^*$-valued Schwartz functions).
  The duality pairing in \eqref{eq:dual-bound} is given in terms of that between $X$ and $X^*$ by
  \begin{equation*}
    \big\langle  \VCarl_{\mf{c}}f(x) ; g(x)\big \rangle = \sum_{j \in \N} \langle \VCarl_{\mf{c},j}f(x) ; g_j(x) \rangle.
  \end{equation*}

  Thanks to the truncated wave packet representation of $\VCarl_{\mf{c}}$ in Corollary \ref{cor:VCarl-WP}, the left hand side of \eqref{eq:dual-bound} can be rephrased in terms of the following wave packet embeddings.
  
\begin{defn}\label{defn:embeddings}
  Let $Y$ be a Banach space and $\phi \in \Sch(\R)$.
  For all $f \in \Sch(\R;Y)$, we define the \emph{wave packet embedding} $\map{\Emb_{\phi}[f]}{\R^3_+}{Y}$ of $f$ with respect to $\phi$ by
\begin{equation*}
  \Emb_{\phi}[f](\eta,y,t):= \langle f ; \Lambda_{(\eta,y,t)}\phi\rangle = \int_{\R} f(x)\,t^{-1}  e^{-2\pi i \eta(x-y)} \bar{\phi}\Bigl( \frac{x-y}{t} \Bigr) \, \dd x.
\end{equation*}
Given families $\Psi^\pm$ of left- and right-truncated wave packets and a measurable function $\map{\mf{c}}{\R}{\Delta}$, for all sequences $g = (g_j)_{j \in \N} \in c_{00}(\N;\Sch(\R;Y))$ we define the \emph{truncated wave packet embeddings with respect to $\mf{c}$}, $\map{\AEmb^{\pm}_{\mf{c}}[g]}{\R^3_+}{Y}$, by 
\begin{equation*}
  \AEmb^{\pm}_{\mf{c}}[g](\eta,y,t) := \int_\RR  \sum_{j \in \N} g_j(x)\overline{ \Lambda_{(\eta,y,t)}\Psi_{(\eta,t)}^{(\mf{c}_{j}(x), \mf{c}_{j+1}(x)), \pm}(x)}\, \dd x.
\end{equation*}

\end{defn}

In \cite[Theorem 5.1]{AU19-2} the following bounds for the wave packet embedding are proven.\footnote{The bounds are proven for a more complex version of the wave packet embedding, with respect to more complex local sizes. The stated embeddings bounds are an immediate consequence. Actually, they are stated for a symmetric inner frequency interval $\Theta_{in}$, but the same proof works for a general interval.}
This is the only place in which the $r_0$-intermediacy assumption is used, and it is not known if this condition is necessary.

\begin{thm}\label{thm:E-embedding}
  Let $\Theta$ and $\Theta_{in}$ be bounded intervals with $0 \in \Theta_{in} \subset \Theta$, and suppose $\phi \in \Sch(\R)$ has Fourier support in $\Theta_{in}$.
  Let $Y$ be an $r_0$-intermediate UMD space for some $r_0 \in [2,\infty)$.
  Then for all $p \in (1,\infty)$ and $q \in (\min(p,r_0)'(r_0-1),\infty]$,
  \begin{equation*}
    \|\Emb_{\phi}[f]\|_{L_\nu^p \sL_\mu^q \FS_{\Theta}} \lesssim \|f\|_{L^p(\R;Y)} \qquad \forall f \in \Sch(\R;Y).
  \end{equation*}
\end{thm}

The main technical goal of this article is to prove bounds for the truncated wave packet embeddings.

\begin{thm}\label{thm:A-embedding}
  Let $Y$ be a UMD Banach space with type $r_0'$ for some $r_0 \in [2,\infty)$.
  For sufficiently small $\mf{b},\epsilon > 0$ and for each choice of sign $\pm$, the following holds.
  Consider the embedding $\AEmb^\pm_{\mf{c}}$ defined with respect to a choice of left/right truncated wave packets $\Psi^\pm$ with parameters $(\mf{b},\epsilon)$.
  Let $\Theta$ be a bounded interval containing $\pm[0,1+\varepsilon)$, and let $\Theta_{in} = \Theta \cap \pm(-\infty,1-\varepsilon)$.
  Then for all $r' \in (1,r_0')$, $p' \in (1,\infty]$, and $q' \in (r', \infty]$, and all measurable $\map{\mf{c}}{\R}{\Delta}$,
  \begin{equation*}
    \|\AEmb^\pm_{\mf{c}}[g]\|_{L_{\nu}^{p'} \sL_{\mu}^{q'} \FS_{\Theta}^*} \lesssim \|g\|_{L^{p'}(\R; \ell^{r'}(\N;Y))}.
  \end{equation*}
\end{thm}

Note that the conditions on the intervals $(\Theta,\Theta_{in})$ are different depending on the choice of sign $\pm$.
These technical conditions are not necessary (it is also possible to prove the result for both embeddings $\AEmb^\pm_{\mf{c}}$ with arbitrary $\Theta_{in} \subset \Theta$ such that $B_{2\mf{b}} \subset \overline{\Theta_{in}} \subset \Theta$, and these conditions are somewhat more natural) but they simplify the proofs considerably.

We now prove Theorem \ref{thm:main-N} as a consequence of Theorem \ref{thm:A-embedding}.
Recall that Theorem \ref{thm:main-N} is an equivalent linearised version of our goal, Theorem \ref{thm:main}.

\begin{proof}[Proof of Theorem \ref{thm:main-N}, assuming Theorem \ref{thm:A-embedding}]
We prove the dual bound \eqref{eq:dual-bound} for all $f \in \Sch(\R;X)$ and all $g \in c_{00}(\N; \Sch(\R;X^*))$.
By the truncated wave packet representation, Corollary \ref{cor:VCarl-WP}, for all sufficiently small $\mf{b},\varepsilon > 0$ there exist $\phi \in \Phi_{\mf{b}}$ and families of left- and right-truncated wave packets $\Psi^\pm$ with parameters $(\mf{b}/2, \varepsilon)$ such that 
\begin{equation}\label{eq:dual-wave-packet-decomposition}
  \int_{\R}\big\langle \VCarl_{\mf{c}}f(x) ; g(x)\big \rangle \, \dd x = \sum_{\square \in \{+,-\}} \int_{\R^{3}_{+}} \big\langle \Emb_{\phi}[f](\eta,y,t) ; \AEmb_{\mf{c}}^{\square}[g](\eta,y,t) \big\rangle \, \dd \eta \, \dd y \, \dd t,
\end{equation}
and thus to prove \eqref{eq:dual-bound} it suffices to prove
\begin{equation}\label{eq:embdual-bd}
  \Big| \int_{\R^{3}_{+}} \big\langle \Emb_{\phi}[f](\eta,y,t); \AEmb_{\mf{c}}^{\pm}[g](\eta,y,t) \big\rangle \, \dd \eta \, \dd y \, \dd t \Big| \lesssim \|f\|_{L^{p}(\R;X)}\|g\|_{L^{p'}(\R;l^{r'}(\N;X^{*}))}.
\end{equation}
Fix a sign $\pm$ and define $(\Theta,\Theta_{in})$ as in Theorem \ref{thm:A-embedding}.
By Corollary \ref{cor:goal-outer-holder}, for all $q \in [1,\infty]$ we have
\begin{equation*}
  \Big| \int_{\R^{3}_{+}} \big\langle \Emb_{\phi}[f](\eta,y,t) ; \AEmb_{\mf{c}}^{\pm}[g](\eta,y,t) \big\rangle \, \dd \eta \, \dd y \, \dd t \Big|
  \lesssim \|\Emb_{\phi}[f]\|_{L_{\nu}^{p} \sL_{\mu}^{q} \FS_{\Theta}} \|\AEmb_{\mf{c}}^{\pm}[g]\|_{L_{\nu}^{p'} \sL_{\mu}^{q'} \FS_{\Theta}^*}.
\end{equation*}
If $q$ satisfies the conditions
\begin{equation}
  \label{eq:q-conditions}
  q > \min(p,r_0)'(r_0-1) \qquad \text{and} \qquad q' > r',
\end{equation}
then provided $\mf{b}$ and $\epsilon$ were chosen sufficiently small, Theorems \ref{thm:E-embedding} and \ref{thm:A-embedding} imply that
\begin{equation*}
  \|\Emb_{\phi}[f]\|_{L_{\nu}^{p} \sL_{\mu}^{q} \FS_\Theta} \|\AEmb_{\mf{c}}^{\pm}[g]\|_{L_{\nu}^{p'} \sL_{\mu}^{q'} \FS_{\Theta}^*}
  \lesssim \|f\|_{L^p(\R;X)} \|g\|_{L^{p'}(\R;\ell^{r'}(\N;X^*))},
\end{equation*}
proving \eqref{eq:embdual-bd} and hence the theorem, so it suffices to show that there exists a $q$ satisfying \eqref{eq:q-conditions}.
This reduces to the inequality
\begin{equation*}
  r > \min(p,r_0)'(r_0-1),
\end{equation*}
which is equivalent to the assumed condition $p > (r/(r_0-1))'$.
\end{proof}

Thus it suffices to prove Theorem \ref{thm:A-embedding}.
In the next section, we reduce this to the domination of the truncated wave packet embeddings $\AEmb_{\mf{c}}^\pm$ by an auxiliary `scalar' embedding.

\subsection{Reduction to domination by a scalar embedding}

\begin{defn}
  Let $Y$ be a Banach space, $\Theta$ a bounded interval containing the origin, and $N \in \N$.
  For all $g \in c_{00}(\N;\Sch(\R;Y))$, all $r \in (1,\infty)$, and all measurable $\map{c}{\R}{\Delta}$, define the auxiliary embedding $\map{\MEmb_{\mf{c},\Theta}^{r',N}[g]}{\R^3_+}{[0,\infty)}$ by 
\begin{equation*}
  \MEmb_{\mf{c},\Theta}^{r',N}[g](\eta,y,t) := \int_{\R} \Big( \sum_{j \in \N} \|g_{j}(x)\|_{Y}^{r'} \1_{\Theta}(t\eta - t\mf{c}_j(x)) \Big)^{1/r'} \, \frac{1}{t} \Big\langle \frac{x-y}{t}  \Big\rangle^{-N} \, \dd x.
\end{equation*}
\end{defn}

Note that this embedding basically ignores the geometry of $Y$, as it only depends on $g$ through the sequence $\|g\|_Y := (\|g_j(x)\|_Y)_{j \in \N}$.
The following bounds follow by applying \cite[Proposition 4.1]{gU16} to the scalar-valued sequence $\|g\|_Y$.

\begin{prop}
  Let $Y$ be a Banach space and $\Theta$ a bounded interval containing the origin.
  If $N \in \N$ is sufficiently large, then for any $r' \in [1,\infty]$, $p' \in (1,\infty]$, and $q' \in (r',\infty]$, and for any function $g \in c_{00}(\N; \Sch(\R;Y))$,
  \begin{equation*}
    \|\MEmb_{\mf{c},\Theta}^{r',N}[g]\|_{L_\nu^{p'} \sL_{\mu}^{q'} \LS_{\Theta}^\infty} \lesssim_{p,q,r} \|g\|_{L^{p'}(\RR; \ell^{r'}(\N;Y))},
  \end{equation*}
  with implicit constant independent of $\mf{c}$ and $N$.
\end{prop}

Thus to prove Theorem \ref{thm:A-embedding}, it suffices to show
\begin{equation*}
  \|\AEmb^\pm_{\mf{c}}[g]\|_{L_{\nu}^{p'} \sL_{\mu}^{q'} \FS_{\Theta}^*} \lesssim \|\MEmb_{\mf{c},\Theta}^{r',N}[g]\|_{L_\nu^{p'} \sL_{\mu}^{q'} \LS_{\Theta}^\infty} \qquad \forall g \in c_{00}(\N; \Sch(\R;Y))
\end{equation*}
for sufficiently large $N \in \N$, under the hypotheses of the theorem.
By definition of the outer Lebesgue quasinorms, this will follow from the following theorem, which extends \cite[Proposition 5.1]{gU16} to the vector-valued setting.
 
\begin{thm}
  \label{thm:main-mass-dom}
  Let $Y$ be a UMD Banach space with type $r_0'$ for some $r_0 \in [2,\infty)$.
  For sufficiently small $\mf{b},\epsilon > 0$ and for each choice of sign $\pm$, the following holds.
  Consider the embedding $\AEmb^\pm_{\mf{c}}$ defined with respect to a choice of left/right truncated wave packets $\Psi^\pm$ with parameters $(\mf{b},\epsilon)$.
  Let $\Theta$ be a bounded interval containing $\pm[0,1+\varepsilon)$, and let $\Theta_{in} = \Theta \cap \pm(-\infty,1-\varepsilon)$.
  Let $E\in\TT_{\Theta}^{\cup}$ be a union of trees.
  Then for all measurable $\map{\mf{c}}{\R}{\Delta}$, all $r \in (1,r_0')$, and all sufficiently large $N \in \N$,
  \begin{equation}
    \label{eq:dom-goal}
    \|\1_{\R_+^3 \sm E} \AEmb_{\mf{c}}^{\pm}[g]\|_{\FS_{\Theta}^*} \lesssim \|\1_{\R_+^3 \sm K} \MEmb_{\mf{c},\Theta}^{r',N}[g] \|_{\LS_{\Theta}^\infty} \qquad \forall g \in c_{00}(\N;\Sch(\R;Y))
  \end{equation}
  with implicit constant independent of $\mf{c}$.
\end{thm}

Thus the only task left---the hardest task---is to prove Theorem \ref{thm:main-mass-dom}.


\section{Scalar domination of the truncated wave packet embedding}
\label{sec:mass-dom}

We devote the entirety of this section to the proof of Theorem \ref{thm:main-mass-dom}.
From now on we only consider the embedding $\AEmb_{\mf{c}}^+$, defined in terms of a fixed family of left-truncated wave packets $\Psi^+$ with sufficiently small parameters $(\mf{b},\varepsilon)$; the proof for $\AEmb_{\mf{c}}^-$ is symmetric.
We fix intervals $\Theta$ and $\Theta_{in}$ as in the statement of the theorem: that is, $\Theta$ is a bounded interval containing $[0,1+\varepsilon)$, and $\Theta_{in} = \Theta \cap (-\infty, 1-\varepsilon)$.
To ease notation we fix $r$, $\mf{c}$, and $N$, and abbreviate $\AEmb_{\mf{c}}^+$, $\MEmb_{\mf{c},\Theta}^{r',N}$, and $\Psi^+$ by $\AEmb$, $\MEmb$, and $\Psi$.

  We need the following useful lemma, which follows from \cite[Lemma 5.5]{gU16}.
  
  \begin{lem}\label{lem:lip-stop-time}
    Let $\map{\tau}{\R}{\R_+}$ be a $1$-Lipschitz function. Then for any non-negative function $\map{h}{\R}{\R_+}$ and any $L>1$ one has
    \begin{equation*}
      \int_\R h(z) \, \dd z \simeq_{L} \int_\R \fint_{B_{\tau(x)/L}(x)} h(z) \, \dd z \, \dd x
    \end{equation*}
    with implicit constant depending only on $L$.
  \end{lem}

  We will be considering sets $T \sm E$, where $T$ is a tree and $E$ is a union of trees.
  The boundary of $E$, viewed in local coordinates with respect to $T$, is encoded as a function $\tau_E$, and the following lemma contains the requisite properties of this function.

  \begin{lem}\label{lem:cover-measure-control}
    Let $E\in\TT^{\cup}$ and fix $T\in\TT$.
    The function
    \begin{equation*}
      \begin{aligned}
        &\tau_{E}(\theta,\zeta) := \sup\big(  \{0\}\cup\{\sigma \colon (\theta,\zeta,\sigma)\in\pi^{-1}_{T}(E)\}\big);
      \end{aligned}
    \end{equation*}
    defines the boundary of $\pi_{T}^{-1}(E)$ as a graph over $(\theta,\zeta) \in \R^2$.
    For every $\theta \in \R$ the map $\zeta\to \tau_{E}(\theta,\zeta)$ is non-negative and $1$-Lipschitz.
    If $\big(\{\theta\}\times\R^{2}_{+}\big) \cap \pi^{-1}_{T}(T\setminus E)$ is non-empty, then $\tau_{E}(\theta,\zeta)\le3+|\zeta|$.
    Furthermore, there exists $E'\in\TT^{\cup}$ with $E'\supseteq E$ such that $T\setminus E=T\setminus E'$ and $\tau_{E'}(\theta,\zeta)> -1 + |\zeta|$.
  \end{lem}
  
  \begin{proof}
    All properties other than the existence of $E'$ are clear from the geometry of trees.
    The set $E'$ can be defined by
    \begin{equation*}
      E'=E\cup \Big( \bigcup_{k\in\Z, |k| \geq 2} D\big(x_{T}+ks_{T},(|k|-1) s_{T}\big) \Big);
    \end{equation*}
    and $\tau_{E'}(\theta,\zeta) > -1 + |\zeta|$ follows.
  \end{proof}

  Let's get to the proof of \eqref{eq:dom-goal}. 
  There are two terms in the local size $\FS_{\Theta}^*$ (see Definition \ref{defn:full-size}) that we need to control: the Lebesgue term, and the randomised term.

  \subsection{The Lebesgue term}
  
  First we prove
  \begin{equation}
    \label{eq:L1-Linfty-control}
    \|\1_{\R_+^3 \sm E} \AEmb[g]\|_{\LS_{\Theta}^{1}(T)} \lesssim \|\1_{\R_+^3 \sm E} \MEmb[g] \|_{\LS_{\Theta}^\infty}
  \end{equation}
  for all trees $T \in \TT_{\Theta}$.
  By translation, dilation, and modulation symmetry we may assume that $T = T(0,0,1)$, and by homogeneity we may assume that the right hand side of \eqref{eq:L1-Linfty-control} is equal to $1$.
  Having made these reductions, it suffices to show
  \begin{equation}
    \label{eq:LS-goal}
    \int_{B_{1}}\int_{0}^{1} \big\| \pi_T^* ( \1_{\R^{3}_{+} \sm E} \AEmb[g] )(\theta,\zeta,\sigma) \big\|_{Y} \frac{\dd \sigma}{\sigma}\, d\zeta  \lesssim c_\theta
  \end{equation}
  for every frequency $\theta \in \Theta \cap \{\theta < 1-\epsilon\}$ such that there exists a point $(\theta,\zeta,\sigma) \in \pi_T^{-1}(T^{in} \sm E)$, with constant $c_\theta$ which is integrable over all such $\theta$.
  From now on we fix such a frequency $\theta$, and we omit it from our notation wherever possible.
  
  The left hand side of \eqref{eq:LS-goal} is controlled by
  \begin{equation}\label{eq:L1-control-1}
     \int_{B_{1}}  \int_{0}^{1} \int_{\R}\sum_{j \in \N} \|g_j(z)\|_{Y} \big|\Lambda_{(\theta\sigma^{-1},0,\sigma)} \Psi_{\sigma}^{(\mf{c}_{j}(z), \mf{c}_{j+1}(z))}(z-\zeta)\big| \, \dd z \,\frac{\dd \sigma}{\sigma}\, \dd\zeta.
   \end{equation}
   By the conditions defining truncated wave packets, $\Psi_{\sigma}^{(\mf{c}_j(z), \mf{c}_{j+1}(z))}$ is non-vanishing only if  $\theta-\sigma \mf{c}_{j}\in B_{\epsilon}(1)$ and $\theta-\sigma\mf{c}_{j+1}
<-1+\epsilon$, thus, since $\theta<1-\epsilon$, the integrand vanishes unless $\mf{c}_{j}<0<\mf{c}_{j+1}$, and 
\begin{equation*}
  \begin{aligned}
    |\Lambda_{(\theta\sigma^{-1},0,\sigma)} \Psi_{\sigma}^{(\mf{c}_j(z), \mf{c}_{j+1}(z))}(z-\zeta) |
    \lesssim\sigma^{-1} \Big\langle \frac{z-\zeta}{\sigma} \Big\rangle^{-N}.
\end{aligned}
\end{equation*}
Thus \eqref{eq:L1-control-1} is bounded by
\begin{equation}\label{eq:L1-control-2}
  \int_{B_{1}} \Big( \int_{\tau(\zeta)}^{4 \tau(\zeta)} + \int_{4 \tau(\zeta)}^{1} \Big) \int_{\R} \|g^{*}(z)\|_{Y} \1_{B_{\epsilon}(1)}(\theta-\sigma\mf{c}^{*}(z))\sigma^{-1} \Big\langle \frac{z-\zeta}{\sigma} \Big\rangle^{-N} \dd z \,\frac{\dd \sigma}{\sigma}\, \dd\zeta
\end{equation}
where $\map{\tau = \tau_{E}(\theta,\cdot)}{\R}{\R_+}$,  $j^{*}(z)\in\N$ is the index (unique if it exists) that satisfies $\mf{c}_{j}(z)<0<\mf{c}_{j+1}(z)$, $g^{*}(z):=g_{j^{*}(z)}(z)$, and $\mf{c}^{*}(z):=c_{j^{*}(z)}(z)$; if no such index $j^{*}$ exists, set $g^{*}(z):=0$, making the other quantities irrelevant. 

Since $\theta - \sigma \mf{c}^*(z) \in B_{\varepsilon}(1)$ and $\sigma > \tau(\zeta)$ implies $\theta - \tau(\zeta) \mf{c}^*(z) \in \Theta$,\footnote{This argument implicitly uses that $B_\varepsilon(1) \subset \Theta$, i.e. that $\Theta$ is larger than $\Theta_{in}$. We will use this argument repeatedly.} the first summand in \eqref{eq:L1-control-2} is controlled by
\begin{equation*} 
  \begin{aligned}
      & \int_{B_{1}}   \int_{\R} \|g^{*}(z)\|_{Y} \1_{\Theta}(\theta-\tau(\zeta)\mf{c}^{*}(z))\tau(\zeta)^{-1} \Big\langle \frac{z-\zeta}{\tau(\zeta)} \Big\rangle^{-N} \, \dd z \, \dd\zeta
\\
&
\lesssim \int_{B_{1}}\MEmb[g](\theta \tau(\zeta)^{-1},\zeta,\tau(\zeta)) \, \dd \zeta \lesssim 1.
    \end{aligned}
  \end{equation*}
Using Lemma \ref{lem:lip-stop-time} we find that the second summand in \eqref{eq:L1-control-2} is controlled by
\begin{equation}\label{eq:L1-control-3}
\int_{B_{1}}  \int_{4 \tau(\zeta)}^{1} \int_{\R}\fint_{B_{\tau(y)/2}(y)} \|g^{*}(z)\|_{Y} \1_{B_{\epsilon}(1)}(\theta-\sigma\mf{c}^{*}(z))
\sigma^{-1} \Big\langle \frac{z-\zeta}{\sigma} \Big\rangle^{-N} \, \dd z \, \dd y \,\frac{\dd \sigma}{\sigma}\, \dd\zeta.
\end{equation}
We split this integral into two parts, according to the conditions $\sigma > \tau(y)$ and $\sigma < \tau(y)$.
In the first case, reasoning as before we have $\theta-\tau(y)\mf{c}^{*}(z)\in \Theta$, so the integral is bounded by
\begin{equation*}
  \begin{aligned}
    &
    \int_{B_{1}}  \int_{4 \tau(\zeta)}^{1}\int_{\R}\fint_{B_{\tau(y)/2}(y)} \|g^{*}(z)\|_{Y} \1_{\Theta}(\theta-\tau(y)\mf{c}^{*}(z)) \, \dd z \\
    &\qquad \times
    \1_{B_{\epsilon}(1)}(\theta-\sigma\mf{c}^{*}(z))\sigma^{-1} \Big\langle \frac{y-\zeta}{\sigma} \Big\rangle^{-N} \, \dd y \, \frac{\dd \sigma}{\sigma} \, \dd\zeta
    \\
    & \lesssim
    \int_{\R}\fint_{B_{\tau(y)/2}(y)} \Big( \int_{4 \tau(\zeta)}^{1}\1_{B_{\epsilon}(1)}(\theta-\sigma\mf{c}^{*}(z)) \, \frac{\dd \sigma}{\sigma} \Big) \sup_{\sigma\in(0,1)}\Big( \int_{B_{1}}  \sigma^{-1} \Big\langle \frac{y-\zeta}{\sigma} \Big\rangle^{-N+2}\, \dd\zeta \Big)
    \\
    &\qquad \times
    \|g^{*}(z)\|_{Y} \1_{\Theta}(\theta-\tau(y)\mf{c}^{*}(z)) \, \dd z \, \langle y \rangle^{-2} \, \dd y 
    \\
    &\lesssim
    \log\Big( \frac{1+\epsilon-\theta}{1-\epsilon-\theta}\Big)
    \int_{\R}\fint_{B_{\tau(y)/2}(y)} 
    \|g^{*}(z)\|_{Y}
    \1_{\Theta}(\theta-\tau(y)\mf{c}^{*}(z))
    \,\dd z\,\langle y \rangle^{-2}\,\dd y
    \\
    &\lesssim
    \log\Big( \frac{1+\epsilon-\theta}{1-\epsilon-\theta}\Big)
    \int_{\R}\MEmb[g](\theta\tau(y)^{-1},y,\tau(y)) \;\langle y \rangle^{-2}\;\dd y \lesssim  c_\theta
  \end{aligned}
\end{equation*}
with $c_\theta$ integrable over $\{\theta \in \Theta : \theta < 1-\epsilon\}$.

In the second case, in which $\sigma < \tau(y)$, the integrand of \eqref{eq:L1-control-3} is non-vanishing only if $\tau(y)/4>\tau(\zeta)$.
This implies $|\tau(y)-\tau(\zeta)|>3 \tau(y)/4$, and since $\tau$ is $1$-Lipschitz we have that $|y-\zeta|>3 \tau(y)/4$.
Since $|z-y|<\tau(y)/2$, we get $4^{-1} |z-\zeta|\leq |y-\zeta|\le4 |z-\zeta|$, and thus $|z-\zeta|>\tau(y)/4$.
Using this, and the trivial bound $\big\langle\frac{y-\zeta}{\sigma}\big\rangle^{-1} \lesssim \langle y\rangle^{-1}$, the integral \eqref{eq:L1-control-3} is controlled by
\begin{equation*}
  \begin{aligned}
    &
    \int_{B_{1}}   \int_{\R}
    \int_{4\tau(\zeta)}^{1}\int_{B_{\tau(y)/2}(y)}   \|g^{*}(z)\|_{Y} \1_{B_{\epsilon}(1)}(\theta-\sigma\mf{c}^{*}(z))\;  \sigma^{-1} \Big\langle \frac{z-\zeta}{\sigma} \Big\rangle^{-N/ 2} \dd z \\
    &\qquad \times  \frac{\sigma}{\tau(y)} \sigma^{-1}\Big\langle \frac{y-\zeta}{\sigma} \Big\rangle^{-N/2+2}  \,\1_{\sigma<\tau(y)} \;\frac{\dd \sigma}{\sigma} \, \langle y\rangle^{-2} \, \dd y \, \dd\zeta \\
    &\lesssim 
    \int_{\R}   
    \int_{0}^{1}\Big(   \int_{B_{1}}  \MEmb[g](\theta\sigma^{-1},\zeta,\sigma) \,\sigma^{-1}\Big\langle \frac{y-\zeta}{\sigma} \Big\rangle^{-N/2+2} \, \1_{4\tau(\zeta)<\sigma} \, \dd\zeta  \Big) \\
    &\qquad \times
    \frac{\sigma}{\tau(y)} \,\1_{\sigma<\tau(y)}
      \frac{\dd \sigma}{\sigma} \, \langle y\rangle^{-2} \, \dd y \lesssim 1,
  \end{aligned}
\end{equation*}
establishing the estimate \eqref{eq:RS-goal}. 

\subsection{Splitting the randomised term}

The randomised term is much harder than the Lebesgue term.
We argue by splitting it into a main term and a number of error terms.
The main term is bounded using variational estimates for convolution operators along with a Littlewood--Paley estimate for UMD spaces; the error terms are bounded either by arguments similar to those used for the Lebesgue term, or by Lemma \ref{lem:short-variation}.

We need to show that
\begin{equation}\label{eq:RS-pregoal}
  \|\1_{\R_+^3 \sm E} \AEmb[g]\|_{\RS(T)} \lesssim \|\1_{\R_+^3 \sm E} \MEmb^{r'}_{\mf{c}}[g] \|_{\LS^\infty}
\end{equation}
for all trees $T$.
As with the Lebesgue term, we may assume that $T = T(0,0,1)$, and that the right hand side of \eqref{eq:RS-goal} is equal to $1$.
  Having made these reductions, it suffices to show
  \begin{equation}
    \label{eq:RS-goal}
    \Big( \int_\R \big\| \pi_T^* ( \1_{\R^{3}_{+} \sm E} \AEmb[g] )(\theta,\zeta,\sigma) \big\|_{\gamma_{d\sigma / \sigma}(\RR_+;Y)}^2 \, d\zeta \Big)^{1/2} \lesssim C_\theta
  \end{equation}
  for every frequency $\theta \in \Theta \cap \{\theta > 1 - \varepsilon\}$ such that there exists a point $(\theta,\zeta,\sigma) \in \pi_T^{-1}(T^{out} \sm E)$, with constant $C_\theta$ which is square-integrable over all such $\theta$.
  From now on we fix such a frequency $\theta$, and we omit it from our notation wherever possible.  As in the previous section, we let $\tau(\cdot) = \tau_E(\theta,\cdot)$.
  According to Lemma \ref{lem:cover-measure-control}  we may assume without loss of generality (by replacing $E$ with a new set $E'$ such that $T \sm E = T \sm E'$) that $-1+|\zeta|<\tau_{E}(\zeta)<3+|\zeta|$.
  
  The left hand side of \eqref{eq:RS-goal} is controlled by
  \begin{equation}\label{eqn:omega-reduction}
    \Big( \int_{B_1} \Big\| \int_\R \sum_{j \in \N} g_j(z) \overline{\tilde{\Psi}_{\sigma}^{(\mf{c}_j(z),\mf{c}_{j+1}(z))}(z-\zeta)} \, \dd z \Big\|_{\gamma_{d\sigma/\sigma}(\tau(\zeta),1;Y)}^2 \, \dd\zeta \Big)^{1/2},
  \end{equation}
  with
  \begin{equation*}   
    \tilde{\Psi}_{\sigma}^{(c_-, c_+)}(z) := \Lambda_{(\theta\sigma^{-1},0,\sigma)}\Psi_{(\theta\sigma^{-1}, \sigma)}^{(c_-, c_+)}.
  \end{equation*}
  We estimate \eqref{eqn:omega-reduction} by duality.
  For all $h \in C^{\infty}_{c}(\mD;Y^{*})$ with $h(\zeta,\sigma)$ vanishing when $\sigma<\tau(\zeta)$ (the smoothness isn't important here) we will show that
  \begin{equation}\label{eq:duality-goal-h}
    \begin{aligned}
      &\Big|  \int_{B_1} \int_{0}^{1} \int_\R \sum_{j \in \N} \Big\langle  g_j(z) \overline{\tilde{\Psi}_{\sigma}^{(\mf{c}_{j}(z),\mf{c}_{j+1}(z))}(z-\zeta)} ; h(\zeta,\sigma) \Big\rangle  \, \dd z \, \frac{\dd \sigma}{\sigma} \, \dd\zeta  \Big| \\
      &\qquad \lesssim C_\theta \|h\|_{L^{2}(B_{1} ; \gamma_{d\sigma/\sigma}(0,1;Y^{*}))}.
    \end{aligned}
  \end{equation}
  Using Lemma \ref{lem:lip-stop-time}, we can control the left hand side of \eqref{eq:duality-goal-h} by
      \begin{equation*}
        \int_{\R} \fint_{B_{\tau(x)/2}(x)} \Big| \sum_{j \in \N} \Big\langle  g_j(z) ; \bigg(  \int_{\tau(x)}^{1} + \int_{0}^{\tau(x)}\bigg) \big(  h(\cdot,\sigma)* \tilde{\Psi}_{\sigma}^{(\mf{c}_{j}(z),\mf{c}_{j+1}(z))}\big)(z) \, \frac{\dd \sigma}{\sigma}  \Big\rangle \Big| \, \dd z \, \dd x.
  \end{equation*}
  The second bracketed term---the one including the integral from $0$ to $\tau(x)$---is our first error term, which we call $\Err^{(1)}$.
  The estimate
  \begin{equation*}
    \Err^{(1)} \lesssim C_\theta \|h\|_{L^{2}(B_{1} ; \gamma_{d\sigma/\sigma}(0,1;Y^{*}))}
  \end{equation*}
  is proven in Section \ref{sec:Err1}.
  As for the first bracketed summand, first note that the integrand vanishes unless $\tau(x) \in (0,1)$, and since $\tau(x) > -1 + |x|$ we can restrict to $x \in B_2$.
  Furthermore, since $\tilde{\Psi}_{\sigma}^{(\mf{c}_{j}(z),\mf{c}_{j+1}(z))}$ vanishes unless $\theta - \sigma \mf{c}_j(z) \in B_\epsilon(1)$, and since $\theta - \sigma \mf{c}_j \in B_{\epsilon}(1)$ and $\sigma > \tau(x)$ implies $\theta - \tau(x)\mf{c}_j \in \Theta$, we have
  \begin{equation}\label{eq:reduce-to-H}
    \begin{aligned}
      &\int_\R \fint_{B_{\tau(x)/2}(x)} \Big| \sum_{j \in \N} \Big\langle  g_j(z) ; \int_{\tau(x)}^{1} \big(h(\cdot,\sigma)* \tilde{\Psi}_{\sigma}^{(\mf{c}_{j}(z),\mf{c}_{j+1}(z))}\big)(z) \, \frac{\dd \sigma}{\sigma} \Big\rangle \Big| \, \dd z \, \dd x \\
      &
      \leq 
      \int_{B_2} \fint_{B_{\tau(x)/2}(x)} \sum_{j \in \N} \|g_j(z)\|_Y 
      \Big\|  \int_{\tau(x)}^{1}\big(h(\cdot,\sigma)* \tilde{\Psi}_{\sigma}^{(\mf{c}_{j}(z),\mf{c}_{j+1}(z))}\big)(z) \,  \frac{\dd\sigma}{\sigma} \Big\|_{Y^*} \, \dd z \, \dd x 
      \\
      &
      \leq \int_{B_2} \fint_{B_{\tau(x)/2}(x)} \Big( \sum_{j \in \N} \|g_j(z)\|_Y^{r'} \1_{\Theta}(\theta - \tau(x)\mf{c}_j(z)) \Big)^{1/r'} \mc{H}_{\tau(x)}(z) \, \dd z \, \dd x 
    \end{aligned} 
  \end{equation}
where for $t > 0$ we write
\begin{equation}\label{eq:reconstructed-H}
  \mc{H}_{t}(z) :=\sup_{\mf{c}\in\Delta}\Big( \sum_{j \in \N}\Big\|\int_{t}^{1} h(\cdot,\sigma)* \tilde{\Psi}_{\sigma}^{(\mf{c}_j(z), \mf{c}_{j+1}(z))}(z) \, \frac{\dd\sigma}{\sigma} \Big\|_{Y^*}^r \Big)^{1/r}.
\end{equation}
The last line of \eqref{eq:reduce-to-H} is bounded by
\begin{equation}\label{eq:bounding-via-H}
  \begin{aligned}
    & \int_{B_2} \MEmb[g](\theta \tau(x)^{-1}, x, \tau(x)) \sup_{z \in B_{\tau(x)/2}(x)} \mc{H}_{\tau(x)}(z) \, \dd x \\
    & \qquad \lesssim \int_{B_2} \sup_{z \in B_{\tau(x)/2}(x)} \mc{H}_{\tau(x)}(z) \dd x.
  \end{aligned}
\end{equation}
We claim that
\begin{equation}\label{eq:MF-control-of-H}
  \begin{aligned}
    &\sup_{z \in B_{\tau(x)/2}(x)} \mc{H}_{\tau(x)}(z) 
    \lesssim (\mc{M} \mc{H}_{*})(x).
  \end{aligned}
\end{equation}
where $\mc{H}_*(z) := \sup_{t \in (0,1)} \mc{H}_t(z)$ and $\mc{M}$ is the Hardy--Littlewood maximal operator.
Assuming \eqref{eq:MF-control-of-H} we have that the last expression in \eqref{eq:bounding-via-H} is controlled by
\begin{equation*} 
    \int_{B_{2}} (\mc{M}\mc{H}_*)(x) \, dx 
    \lesssim \| \mc{M}\mc{H}_* \|_{L^2(\R)} 
    \lesssim \|\mc{H}_*  \|_{L^2(\R)}.
\end{equation*}
Thus, assuming \eqref{eq:MF-control-of-H}, it suffices to show
\begin{equation}\label{eq:Hstar-L2-bound}
    \|\mc{H}_*\|_{L^2(\R)} \lesssim C_\theta \|h\|_{L^2(B_{1} ; \gamma_{d\sigma/\sigma}(0,1;Y^{*}))}.
\end{equation}

To show \eqref{eq:MF-control-of-H} notice that the Fourier support of $\tilde{\Psi}_{\sigma}^{(\mf{c}_j(z), \mf{c}_{j+1}(z))}$ is contained in $(0,|\Theta|+\mf{b})$.
Let $\Upsilon\in\Sch(\R)$ be such that $\FT{\Upsilon}=1$ on $[0,|\Theta|+\mf{b}]$ and set $\Upsilon_{t}=\Dil_{t}\Upsilon$.
It then holds that
\begin{equation*}
  \begin{aligned}
    \mc{H}_{t}(z) &=\sup_{\mf{c}\in\Delta}\Big( \sum_{j \in \N}\Big\| \int_{B_{1}} \int_{t}^{1} h(\zeta,\sigma)\; \tilde{\Psi}_{\sigma}^{(\mf{c}_j, \mf{c}_{j+1})}* \Upsilon_{t}(z-\zeta)  \, \dd\zeta \, \frac{\dd\sigma}{\sigma} \Big\|_{Y^*}^r \Big)^{1/r}
    \\
    &
    \leq \int_{\R}\sup_{\mf{c}\in\Delta}\Big( \sum_{j \in \N}\Big\| \int_{B_{1}} \int_{t}^{1} h(\zeta,\sigma)\; \tilde{\Psi}_{\sigma}^{(\mf{c}_j, \mf{c}_{j+1})} (z'-\zeta) \, \dd\zeta \, \frac{\dd\sigma}{\sigma} \Big\|_{Y^*}^r \Big)^{1/r} \Upsilon_{t}(z-z') \, \dd z'
    \\
    &
    =  \int_{\R}\mc{H}_{t}(z') \Upsilon_{t}(z-z')\dd z' 
    \lesssim \inf_{z'\in B_{t}(z)}(\mc{M}\mc{H}_{*})(z'),
\end{aligned}
\end{equation*}
which implies \eqref{eq:MF-control-of-H}. 

\subsection{Controlling $\mc{H}_{*}$}

To estimate $\mc{H}_{*}=\sup_{t\in(0,1)}\mc{H}_{t}$, we compare it with the variation of a family of convolution operators applied to a function constructed from $h$.

The truncated wave packet $\tilde{\Psi}^{(c_-, c_+)}_{\sigma}$ is non-vanishing only if
\begin{equation*}
   \theta\in \big(  \sigma c_- +(1-\epsilon), \sigma c_+ - (1-\epsilon) \big) \cap \big(  \sigma c_- + (1-\epsilon),\sigma c_- +(1+\epsilon)\big).
\end{equation*}
Since $\theta>1-\epsilon$, $\tilde{\Psi}^{(c_-, c_+)}_{\sigma}$ is non-vanishing only if $c_+>0$.
This means that in \eqref{eq:reconstructed-H} we need only consider $\mf{c}\in\Delta$ such that $\mf{c}_{j+1}>0$ for all $j\in\N$.
Henceforth we fix $t\in(0,1)$ and $z \in \R$, and we bound $\mc{H}_{t}(z)$ pointwise.

For any $\mf{c}\in\Delta$ we define
\begin{equation}\label{eq:sigma-defn}
  \begin{aligned}
    \sigma_{\mf{c},j}^{-} &:= \inf \{\sigma \in (t,1) :  \tilde{\Psi}_{\sigma}^{(\mf{c}_j, \mf{c}_{j+1})} \neq 0\} \\
    \sigma_{\mf{c},j}^{+} &:= \sup \{\sigma \in (t,1) : \tilde{\Psi}_{\sigma}^{(\mf{c}_j, \mf{c}_{j+1})} \neq 0\}
  \end{aligned}
\end{equation}
with $\sigma^{\pm}_{\mf{c},j}$ implicitly depending on $t\in(0,1)$; let us drop the indices $j\in\N$ for which the interval $(\sigma_{\mf{c},j}^{-}, \sigma_{\mf{c},j}^{+})$ is empty.
We have  
\begin{equation}\label{eq:sigmapm-bounds}
  \sigma_{\mf{c},j}^{-}\geq \frac{\theta+(1-\epsilon)}{\mf{c}_{j+1}} \quad \forall j\in\N, \qquad \sigma_{\mf{c},j+1}^{+} \leq \frac{\theta-(1-\epsilon)}{\mf{c}_{j+1}} \quad \forall j\in\N
\end{equation}
and thus
\begin{equation}\label{eq:sigmapm-relation}
\frac{\sigma_{\mf{c},j+1}^{+}}{\sigma_{\mf{c},j}^{-}}\leq \frac{\theta-(1-\epsilon)}{\theta+(1-\epsilon)} < 1 - 2\frac{(1-\epsilon)}{\theta + (1-\epsilon)}.
\end{equation}
It follows that for any $\lambda>1$ the intervals  $(\lambda^{-1}\sigma_{\mf{c},j}^{-}, \lambda \sigma_{\mf{c},j}^{+})$ have finite overlap, so that
\begin{equation}
  \label{eq:interval-overlap}
\Big\|\sum_{j\in\N} \1_{(\lambda^{-1}\sigma_{\mf{c},j}^{-}, \lambda \sigma_{\mf{c},j}^{+} )}\Big\|_{L^{\infty}}\lesssim_{\lambda, |\Theta|, \epsilon} 1
\end{equation}
with implicit constant independent of $\theta$ and $\mf{c}$.
Fix a large $\lambda > 1$ to be determined later.

By construction we have 
\begin{equation*}
  \mc{H}_{t}(z) =\sup_{\mf{c}\in\Delta} \Big( \sum_{j \in \N} \Big\| \int_{\sigma_{\mf{c},j}^{-}}^{\sigma_{\mf{c},j}^{+}}h(\cdot,\sigma)* \tilde{\Psi}_{\sigma}^{(\mf{c}_j, \mf{c}_{j+1})}(z)\,  \frac{\dd \sigma}{\sigma}   \Big\|_{Y^*}^r \Big)^{1/r};
\end{equation*}
don't forget that the endpoints $\sigma^{\pm}_{\mf{c},j}$ depend implicitly on $t$.
We bound $\mc{H}_t$ pointwise by
\begin{equation*}
    \mc{H}_{t}(z)
    \lesssim \mc{G}(z) + \Err_{t}^{(2)}(z)
\end{equation*}
where
\begin{equation*}
  \Err_{t}^{(2)}(z) := \sup_{\mf{c}\in\Delta}\Big( \sum_{j \in \N} \Big\|  \int_{\sigma_{\mf{c},j}^{-}}^{\sigma_{\mf{c},j}^{+}}h(\cdot,\sigma) * (\tilde{\Psi}_{\sigma}^{(\mf{c}_j, \mf{c}_{j+1})} - \tilde{\Psi}_{\sigma})(z) \, \frac{\dd \sigma}{\sigma} \Big\|_{Y^*}^r \Big)^{1/r}
\end{equation*}
with
\begin{equation*}
  \tilde{\Psi}_{\sigma}(z) := \Dil_{\sigma}\Mod_{\theta}\Psi_{(\theta\sigma^{-1},\sigma)}^{(0,+\infty)}(z),
\end{equation*}
and
\begin{equation*}
  \mc{G}(z)
  :=
  \sup_\sigma \Big( \sum_{j \in \N} \Big\| \int_{\sigma_{j}^{-}}^{\sigma_{j}^{+}} h(\cdot,\sigma)* \tilde{\Psi}_{\sigma}(z)\frac{\dd \sigma}{\sigma} \Big\|_{Y^*}^r \Big)^{1/r}
\end{equation*}
with supremum taken over all sequences $\sigma = (\sigma^\pm_j)_{j \in \N}$ satisfying
\begin{equation}
  \label{eq:sigma-conditions}
  0 < \cdots < \sigma_{j+1}^+ < \sigma_{j}^- < \sigma_{j}^+ < \sigma_{j-1}^- < \cdots < 1
  \quad \text{and} \quad
  \Big\|\sum_{j \in \N} \1_{(\lambda^{-1} \sigma_j^- , \lambda \sigma_j^+)} \Big\|_{L^\infty} \leq C
\end{equation}
for $C$ larger than the implicit constant in \eqref{eq:interval-overlap}.
The term $\mc{G}$ is further decomposed by introducing the family of functions
\begin{equation*}
  H_{\rho}(z) := \int_{0}^{\rho} h(\cdot,\sigma)* \tilde{\Psi}_{\sigma}(z) \, \frac{\dd \sigma}{\sigma},
\end{equation*}
so that
\begin{equation*}
  \mc{G}(z) = \sup_\sigma \Big( \sum_{j \in \N} \big\| H_{\sigma_{j}^{+}}(z) - H_{\sigma_{j}^{-}}(z) \big\|_{Y^*}^r \Big)^{1/r}.
\end{equation*}
We will compare $\mc{G}$ with the variation of a family of convolution operators applied to $H_{1}$.
In particular for a fixed $\Upsilon \in \Sch(\RR)$ (to be chosen later) we obtain the pointwise bound
\begin{equation*}
  \mc{G}(z) \leq \|H_{1}* \Upsilon_{\sigma}(z)\|_{V^{r}(\sigma\in(0,1);Y^{*})} + \Err^{(3)}(z)
\end{equation*}
with $\Upsilon_{\sigma}:=\Dil_{\sigma}\Upsilon$ and
\begin{equation*}
  \Err^{(3)}(z) := \sup_{\sigma}\Big( \sum_{j \in \N} \big\| (H_{\sigma_{j}^+}(z)- H_{1}\ast \Upsilon_{\sigma_{j}^{+}}(z)) - (H_{\sigma_{j}^{-}}(z) -H_1 \ast \Upsilon_{\sigma_{j}^{-}}(z)) \big\|_{Y^*}^r \Big)^{1/r}.
\end{equation*}
The Fourier support of $\tilde{\Psi}_{\sigma}$ is contained in $B_{\mf{b}/\sigma}(\theta/\sigma)$, so we choose $\Upsilon$ so that $\spt \FT{\Upsilon}\subset B_{\theta+2\mf{b}}$ and $\FT{\Upsilon}=1$ on $B_{\theta+\mf{b}}$.

Now we have the pointwise estimate
\begin{equation*}
  \mc{H}_{t}(z)
  \lesssim  \|\sigma \mapsto H_{1}* \Upsilon_{\sigma}(z)\|_{V^{r}(0,1;Y^{*})}  + \Err_{t}^{(2)}(z) + \Err^{(3)}(z)
\end{equation*}
and it remains to show
\begin{align}
 \label{eq:main-term-bound}   &\big\| \|\sigma \mapsto H_{1}* \Upsilon_{\sigma} \|_{V^{r}(0,1;Y^{*})} \big\|_{L^2(\R)} \lesssim C_\theta \|h\|_{L^2(B_{1} ; \gamma_{d\sigma/\sigma}(0,1;Y^{*}))}
      \\
  \label{eq:err2-term-bound}
                              &  \| \Err_{t(z)}^{(2)}(z)\|_{L_{\dd z}^{2}(\R)} \lesssim C_\theta \|h\|_{L^2(B_{1} ; \gamma_{d\sigma/\sigma}(0,1;Y^{*}))}
  \\
   \label{eq:err3-term-bound}
&
  \| \Err^{(3)}\|_{L^2(\R)} \lesssim C_\theta \|h\|_{L^2(B_{1} ; \gamma_{d\sigma/\sigma}(0,1;Y^{*}))}
\end{align}
for any measurable function $\map{t}{\R}{(0,1)}$.

\subsection{Bounding the variation term}
Thanks to Theorem \ref{thm:cotype-variational-convolution} (valid since $Y^*$ has martingale cotype $r_0 < r$) we have
\begin{equation*}
  \big\| \|\sigma \mapsto H_{1}* \Upsilon_{\sigma} \|_{V^{r}(0,1;Y^{*})} \big\|_{L^2(\R)} \lesssim \|H_{1} \|_{L^2(\R; Y^{*})},
\end{equation*}
so to prove \eqref{eq:main-term-bound} it remains to show the bound
\begin{equation*}
\|  H_{1}\|_{L^{2}(\R;Y^{*})}\lesssim C_\theta  \|h\|_{L^2(B_{1} ; \gamma_{d\sigma/\sigma}(0,1;Y^{*}))}.
\end{equation*}
We estimate by duality by fixing $\mf{H} \in L^2(\RR;Y)$ and writing
\begin{equation*}
  \begin{aligned}
    &\Big| \int_{\R} \langle H_{1}(z) ; \mf{H}(z) \rangle \, \dd z \Big| \\
    &= \Big|  \int_{\tau}^{1}\int_{B_{1}} \Big\langle h(\zeta,\sigma)   ; \int_{\R}\mf{H}(z) \overline{\tilde{\Psi}_{\sigma}(z-\zeta)} \, \dd z\Big\rangle \, \dd \zeta \, \frac{\dd \sigma}{\sigma}  \Big| \\
    &\lesssim \|h\|_{L^2(B_{1} ; \gamma_{d\sigma/\sigma}(0,1;Y^*))} \Big\| \int_{\RR} \mf{H}(z) \overline{\tilde{\Psi}_{\sigma}(z-\zeta)} \, \dd z \Big\|_{L_\zeta^2 (B_{1} ; \gamma_{d\sigma/\sigma}(0,1;Y))}.
  \end{aligned}
\end{equation*}
By the frequency-scale compatibility condition for $\Psi^{(0,+\infty)}$ we have
\begin{equation*}
  \begin{aligned}
    \tilde{\Psi}_{\sigma}
    &= \Dil_{\sigma}\Mod_{\theta}\Psi^{(0,+\infty)}_{(\theta\sigma^{-1}, \sigma)} \\
    &= \Dil_{\sigma}\Mod_{\theta}\Psi^{(0,+\infty)}_{(\theta,1)}.
\end{aligned}
\end{equation*}
Since $\Mod_{\theta}\Psi^{(0,+\infty)}_{(\theta,1)}$ is uniformly bounded in arbitrary Schwartz seminorms and has  Fourier support in $B_{\mf{b}}(\theta)$ (and in particular has vanishing mean), by the continuous Littlewood--Paley estimate for UMD spaces (Proposition \ref{prop:UMD-littewood-paley}) we have
\begin{equation*}
  \Big\| \int_{\RR} \mf{H}(z) \overline{\tilde{\Psi}_{\sigma}(z-\zeta)} \, \dd z \Big\|_{L_{\zeta}^{2}(B_{1} ; \gamma_{d\sigma/\sigma}(0,1;Y))} \lesssim \|\mf{H}\|_{L^2(B_{1};Y)}
\end{equation*}
as required (here there is no dependence on $\theta$ in the constant, since $\mf{b}$ is assumed to be sufficiently small).

\subsection{Bounding the first error term $\Err^{(1)}$.}\label{sec:Err1}

We show that
\begin{equation*}
  \begin{aligned}
    \Err^{(1)} = \int_\R \fint_{B_{\tau(x)/2}(x)} \Big| \sum_{j \in \N} \Big\langle  g_j(z) ; &\int_{B_1} \int_{0}^{\tau(x)}  h(\zeta,\sigma) \tilde{\Psi}_{\sigma}^{(\mf{c}_j(z), \mf{c}_{j+1}(z))}(z-\zeta) \, \frac{\dd \sigma}{\sigma} \dd\zeta \Big\rangle \Big| \, \dd z \, \dd x \\
    &\lesssim C_\theta \|h\|_{L^2(B_1; \gamma_{\dd\sigma/\sigma}(0,1;Y^*))}.
  \end{aligned}
\end{equation*}
Decompose $\Err^{(1)}$ into the sum of two terms, $\Err^{(1,1)}$ and $\Err^{(1,2)}$, by splitting the integral in $\sigma$ according to the conditions $\sigma < 4\tau(\zeta)$ (yielding $\Err^{(1,1)}$) and $\sigma > 4\tau(\zeta)$ (yielding $\Err^{(1,2)}$).
By Lemma \ref{lem:gamma-L1-emb} and some rearranging, $\Err^{(1,1)}$ is controlled by
\begin{equation*}
    \begin{aligned}[t]
   \int_{B_{1}}  \int_{\R}\sum_{j \in \N}\|g_{j}(z)\|_{Y}\Big(   \int_{\tau(\zeta)}^{4\tau(\zeta)} & |\1_{B_{\epsilon}(1)}(\theta-\sigma \mf{c}_{j}(z))\1_{(-\infty,-1+\epsilon)}(\theta-\sigma \mf{c}_{j+1}(z))|^{2}\frac{\dd \sigma}{\sigma}\Big)^{\frac{1}{2}}
    \\
    &\times
    \tau(\zeta)^{-1}\Big\langle\frac{z-\zeta}{\tau(\zeta)}\Big\rangle^{-N}\|h(\zeta,\sigma)\|_{\gamma_{\dd\sigma/\sigma}(0,1;Y^{*})} \, \dd z \, \dd \zeta.
  \end{aligned}
\end{equation*}
For fixed $\zeta \in B_1$ and $z \in \R$ the integrand vanishes unless  $\mf{c}_{j+1}(z)>0$ for all $j\in\N$.
Furthermore, for any fixed $\zeta,z\in B_{1}\times \R$ there are at most finitely many indices $j\in\N$ (with quantity independent of $\theta > 1-\epsilon$) such that the integrand above doesn't vanish; this follows from the fact that if $\1_{(-\infty,-1+\epsilon)}(\theta-\sigma \mf{c}_{j+1}(z))\ne 0 $ for some $\sigma\in(0,1)$ then for any $j'\geq j+1$ one has
$\1_{B_{\epsilon}(1)}(\theta-\sigma' \mf{c}_{j'}(z))=0$ unless $\sigma'< \sigma \frac{\theta-(1-\epsilon)}{\theta+(1-\epsilon)}<\sigma(1- \frac{1}{2|\Theta|})$ (assuming $\epsilon>0$ is small enough).
Since $\theta-\sigma\mf{c}_{j}(z)\in B_{\epsilon}(1)\subset \Theta$ and $\sigma>\tau(\zeta)$ implies $\theta-\tau(\zeta)\mf{c}_{j}(z)\in \Theta$, we may control $\Err^{(1,1)}$ by
 \begin{equation*}
  \begin{aligned}
   \int_{B_{1}}  &\int_{\R}\Big(  \sum_{j \in \N}\|g_{j}(z)\|_{Y}^{r}\1_{\Theta}(\theta-\tau(\zeta) \mf{c}_{j}(z))\Big)^{\frac{1}{r}}
   \tau(\zeta)^{-1}\Big\langle\frac{z-\zeta}{\tau(\zeta)}\Big\rangle^{-N}\dd z
   \\
    &\times\|h(\zeta,\sigma)\|_{\gamma_{\dd\sigma/\sigma}(0,1;Y^{*})}  \dd \zeta \lesssim \|h(\zeta,\sigma)\|_{L^{2}(B_{1};\gamma_{\dd\sigma/\sigma}(0,1;Y^{*}))} 
  \end{aligned}
\end{equation*}
as required.

By Lemma \ref{lem:gamma-L1-emb} we can control $\Err^{(1,2)}$ by 
\begin{equation*}
  \begin{aligned}[t]
      \int_\R&\fint_{B_{\tau(x)/2}(x)} \int_{B_1}   
      \|h(\zeta,\sigma)\|_{\gamma_{\dd\sigma/\sigma}(0,1;Y)}
      \\
      &
      \times
      \Big( \int_{4\tau(\zeta)}^{\tau(x)}\Big|\sum_{j \in \N}  \|g_j(z)\|_{Y^{*}}  |\tilde{\Psi}_{\sigma}^{(\mf{c}_{j}(z), \mf{c}_{j+1}(z))}(z-\zeta)|  \Big|^{2} \,\frac{\dd \sigma}{\sigma} \Big)^{1/2}\dd\zeta   \, \dd z \, \dd x.
  \end{aligned}
\end{equation*}
For any $z\in\R$ and $\sigma\in(0,1)$ there is at most one $j\in\N$ such that $\tilde{\Psi}_{\sigma}^{(\mf{c}_{j}(z), \mf{c}_{j+1}(z))}$ is non-vanishing.
Furthermore $\tilde{\Psi}_{\sigma}^{(\mf{c}_{j}(z), \mf{c}_{j+1}(z))}(z-\zeta) $ vanishes unless $\theta-\sigma\mf{c}_{j}(z)\in B_{\epsilon}(1)\subset \Theta$, so we can control $\Err^{(1,2)}$ by 
\begin{equation}\label{eq:E12-ctrl-1}
  \begin{aligned}
    &\int_{B_1}   
    \|h(\zeta,\sigma)\|_{\gamma_{\dd\sigma/\sigma}(0,1;Y)} \int_\R\fint_{B_{\tau(x)/2}(x)} \\
    &\qquad  \Big( \sum_{j \in \N}  \|g_j(z)\|_{Y^{*}}^{2} 
    \int_{4\tau(\zeta)}^{\tau(x)}\1_{B_{\epsilon}(1)}(\theta-\sigma\mf{c}_{j}(z)) |\tilde{\Psi}_{\sigma}^{(\mf{c}_{j}(z),\mf{c}_{j+1}(z))}(z-\zeta)|^{2}\,\frac{\dd \sigma}{\sigma} \Big)^{1/2} \, \dd z \, \dd x\,  \dd\zeta .
  \end{aligned}
\end{equation}
Here we are integrating over $(x,\zeta)\in \R\times B_{1}$ such that $\tau(x)>4\tau(\zeta)$.
This implies that $|\tau(x)-\tau(\zeta)|>3\tau(x)/4$, and since $\tau$ is $1$-Lipschitz we have that $|x-\zeta|>3\tau(x)/4$.
Since $|z-x|<\tau(x)/2$,
we get  $4^{-1}|z-\zeta|\leq |x-\zeta|\leq 4 |z-\zeta|$, and hence $|z-\zeta|>\tau(x)/4$.
This yields
\begin{equation}\label{eq:E12-psi-ctrl}
  |\tilde{\Psi}_{\sigma}^{(\mf{c}_{j}(z),\mf{c}_{j+1}(z))}(z-\zeta)|
  \lesssim
  \frac{\sigma^2}{\tau(x)^{2}}\sigma^{-1}\Big\langle \frac{z-\zeta}{\sigma}\Big\rangle^{-N+2}.
\end{equation}
Let $A(\zeta)$ denote the integral in $(x,z)$ appearing in \eqref{eq:E12-ctrl-1}. 
Using \eqref{eq:E12-psi-ctrl} and the Minkowski inequality, for $\zeta \in B_1$ we can control $A(\zeta)$ by
\begin{equation*}
  \begin{aligned}
    &\int_\R\fint_{B_{\tau(x)/2}(x)} \int_{4\tau(\zeta)}^{\tau(x)} \frac{\sigma}{\tau(x)^{2}} \\
    & \qquad \Big( \sum_{j \in \N}  \|g_j(z)\|_{Y^{*}}^{2} \1_{B_{\epsilon}(1)}(\theta-\sigma\mf{c}_{j}(z)) \Big)^{1/2} \Big\langle \frac{z-\zeta}{\sigma}\Big\rangle^{-N+2}   \,\frac{\dd \sigma}{\sigma} \, \dd z \, \dd x
    \\
     &
     \lesssim
     \int_\R\langle x\rangle^{-2}\fint_{B_{\tau(x)/2}(x)} \int_{4\tau(\zeta)}^{\tau(x)} \frac{\sigma^{2}}{\tau(x)^{2}} \\
     & \qquad \Big( \sum_{j \in \N}  \|g_j(z)\|_{Y^{*}}^{2} \1_{B_{\epsilon}(1)}(\theta-\sigma\mf{c}_{j}(z)) \Big)^{1/2}      \sigma^{-1} \Big\langle \frac{z-\zeta}{\sigma}\Big\rangle^{-N+4}   \,\frac{\dd \sigma}{\sigma} \, \dd z \,  \dd x.
  \end{aligned}
\end{equation*}
Thus we can control $\Err^{(1,2)}$ by
\begin{equation*}
  \begin{aligned}
    &
    \int_{B_1}   
    \|h(\zeta,\sigma)\|_{\gamma_{\dd\sigma/\sigma}(0,1;Y)}\int_\R\langle x\rangle^{-2}\int_{4\tau(\zeta)}^{\tau(x)} \frac{\sigma^{3}}{\tau(x)^{3}}    \\
    & \qquad \times
    \int_{\R} \Big( \sum_{j \in \N}  \|g_j(z)\|_{Y^{*}}^{2}  \1_{B_{\epsilon}(1)}(\theta-\sigma\mf{c}_{j}(z)) \Big)^{1/2}      \sigma^{-1}\Big\langle \frac{z-\zeta}{\sigma}\Big\rangle^{-N/2} \dd z \\
    & \qquad \times
    \sigma^{-1} \Big\langle \frac{x-\zeta}{\sigma}\Big\rangle^{-N/2+4}   \,\frac{\dd \sigma}{\sigma} \, \dd x \, \dd\zeta \\
    & \lesssim 
    \int_{B_1} \|h(\zeta,\sigma)\|_{\gamma_{\dd\sigma/\sigma}(0,1;Y)}\int_\R\langle x\rangle^{-2}\int_{4\tau(\zeta)}^{\tau(x)} \frac{\sigma^{3}}{\tau(x)^{3}} \\
    & \qquad \times
    \MEmb[g](\theta\sigma^{-1},\zeta,\sigma)  \sigma^{-1} \Big\langle \frac{x-\zeta}{\sigma} \Big\rangle^{-N/2+4}   \,\frac{\dd \sigma}{\sigma} \, \dd x \, \dd\zeta \\
    & \lesssim
    \int_\R\langle x\rangle^{-2}  \int_{0}^{\tau(x)} \frac{\sigma^{3}}{\tau(x)^{3}}    \,\frac{\dd \sigma}{\sigma}  \\
    &\qquad \times
    \sup_{\sigma>4\tau(\zeta)} \int_{B_{1}}\|h(\zeta,\sigma)\|_{\gamma_{\dd\sigma/\sigma}(0,1;Y^*)} \sigma^{-1} \Big\langle \frac{x-\zeta}{\sigma} \Big\rangle^{-N/2+4}  \, \dd\zeta \, \dd x \\
    & \lesssim
    \int_\R\langle x\rangle^{-2} \mc{M}\big(\|h(\cdot,\sigma)\|_{\gamma_{\dd\sigma/\sigma}(0,1;Y^*)}\big)(x) \,  \dd x \\
    & \lesssim 
    \|h(\zeta,\sigma)\|_{L^{2}_{\dd\zeta}(\gamma_{\dd\sigma/\sigma}(0,1;Y^*))}
  \end{aligned}
\end{equation*}
as required.

\subsection{Bounding the second error term $\Err_{t(\cdot)}^{(2)}$}\label{sec:Err2}
For each measurable function $\map{t}{\R}{(0,1)}$ we show the bound \eqref{eq:err2-term-bound}.
Fix a measurable function $\map{\mf{c}}{\R}{\Delta}$ and let
\begin{equation*}
  P_{\sigma}^{j,z} := \tilde{\Psi}_{\sigma}^{(\mf{c}_{j}(z),\mf{c}_{j+1}(z))} - \tilde{\Psi}_{\sigma} \in \Sch(\R),
\end{equation*}
 so that for all $z \in \R$,
\begin{equation*}
  \Err_{t(z)}^{(2)}(z) \leq \Big( \sum_{j \in \N} \Big\| \int_{\sigma_{\mf{c},j}^{-}}^{\sigma_{\mf{c},j}^{+}} (h(\cdot,\sigma)* P_{\sigma}^{j,z})(z)  \, \frac{\dd \sigma}{\sigma} \Big\|_{Y^*}^r \Big)^{1/r}
\end{equation*}
where the sequences $(\sigma_{\mf{c},j}^\pm)_{j \in \N}$, constructed in \eqref{eq:sigma-defn}, implicitly depend on $z$ and $t(z)$.
By Lemma \ref{lem:short-variation}, using that $Y^*$ is UMD with cotype $r$ and that the intervals $I_{j,z} := (\sigma_{\mf{c},j}^{-}(z), \sigma_{\mf{c},j}^+(z))$ have finite overlap, for any measurable $\map{q}{\N \times \R \times \R_+}{(0,\infty)}$ we have
\begin{equation*}
  \begin{aligned}
    &\| \Err_{t(\cdot)}^{(2)} \|_{L^2(\R)} \\
    &\lesssim \|h\|_{L^2(B_1 ; \gamma(0,1;Y^*))} \sup_{j,z,\sigma} \Big\| \frac{\Dil_{\sigma^{-1}} P_{\sigma}^{j,z}}{q_{j,z}(\sigma)} \Big\|_* \sup_{j,z} \|q_{j,z}(\sigma)\|_{L^{2}_{\dd\sigma/\sigma}(I_{j,z})}
  \end{aligned}
\end{equation*}
where $\| \cdot \|_*$ is a Schwartz seminorm of sufficiently high order.
We claim that
\begin{equation}
  \label{eq:Psigma-est}
\|  \Dil_{\sigma^{-1}}  P_{\sigma}^{j,z}\|_{*}\lesssim |\sigma \mf{c}_{j}(z)| + \max(3+\theta -  \sigma\mf{c}_{j+1}(z),0);
\end{equation}
setting $q_{j,z}(\sigma)$ to be the right hand side of \eqref{eq:Psigma-est}, it will then suffice to show
\begin{equation}
  \label{eq:q-est}
  \sup_{j,z} \|q_{j,z}(\sigma)\|_{L^2_{\dd\sigma/\sigma}(I_{j,z})} \lesssim C_\theta,
\end{equation}
with a constant $C_\theta$ that is square-integrable over $\{\theta \in \Theta : \theta > 1-\epsilon\}$.
We claim that this estimate also holds.

We are left with proving the claimed estimates.
Let's start with \eqref{eq:Psigma-est}.
Writing out the definitions, we see that
\begin{equation*}
  \Dil_{\sigma^{-1}} P_{\sigma}^{j,z}
  = \Mod_{\theta} \big( \Psi^{(\mf{c}_j(z), \mf{c}_{j+1}(z))}_{(\theta\sigma^{-1}, \sigma)} - \Psi^{(0,\infty)}_{(\theta\sigma^{-1}, \sigma)} \big),
\end{equation*}
so we can estimate
\begin{equation*}
  \begin{aligned}
    \|\Dil_{\sigma^{-1}} P_{\sigma}^{j,z}\|_*
    &\lesssim \big\| \Psi^{(\mf{c}_j(z), \mf{c}_{j+1}(z))}_{(\theta\sigma^{-1}, \sigma)} - \Psi^{(0,\infty)}_{(\theta\sigma^{-1}, \sigma)} \|_* \\
    &\leq \big\| \Psi^{(\mf{c}_j(z), \mf{c}_{j+1}(z))}_{(\theta\sigma^{-1}, \sigma)} - \Psi^{(0, \mf{c}_{j+1}(z))}_{(\theta\sigma^{-1}, \sigma)} \big\|_* + \big\| \Psi^{(0,\infty)}_{(\theta\sigma^{-1}, \sigma)} - \Psi^{(0, \mf{c}_{j+1}(z))}_{(\theta\sigma^{-1}, \sigma)} \big\|_*.
  \end{aligned}
\end{equation*}
Use the smoothness condition of families of left-truncated wave packets to bound the first summand by
\begin{equation*}
  \sigma \int_0^{\mf{c}_{j}(z)} \big\| (\sigma^{-1} \partial_{c_-}) \Psi^{(\lambda, \mf{c}_{j+1}(z))}_{(\theta\sigma^{-1}, \sigma)} \big\|_* \, \dd \lambda
  \leq |\sigma \mf{c}_j(z)|.
\end{equation*}
As for the second summand, note by weak dependence on the right endpoint that the second summand vanishes if $\mf{c}_{j+1}(z) > \theta\sigma^{-1} + 3\sigma^{-1}$.
Assuming this is not the case, the second summand can be written as
\begin{equation*}
  \begin{aligned}
    \big\| \Psi^{(0,\theta\sigma^{-1} + 3\sigma^{-1})}_{(\theta\sigma^{-1}, \sigma)} - \Psi^{(0, \mf{c}_{j+1}(z))}_{(\theta\sigma^{-1}, \sigma)} \big\|_*
    &\leq \sigma \int_{\mf{c}_{j+1}(z)}^{\theta\sigma^{-1} + 3\sigma^{-1}} \big\| (\sigma^{-1} \partial_{c_+}) \Psi^{(0,\lambda)}_{(\theta\sigma^{-1}, \sigma)} \big\|_* \, \dd\lambda \\
    &\lesssim 3 + \theta - \sigma\mf{c}_{j+1}(z),
  \end{aligned}
\end{equation*}
again using the smoothness condition.
This proves \eqref{eq:Psigma-est}.

Finally we prove \eqref{eq:q-est}.
For fixed $j$ and $z$ we have
\begin{equation*}
  \begin{aligned}
  \int_{\sigma_{\mf{c},j}^{-}(z)}^{\sigma_{\mf{c},j}^{+}(z)} |\sigma \mf{c}_j(z)|^2 \, \frac{\dd\sigma}{\sigma}
  &\leq |\mf{c}_j(z)|^2 \int_{\sigma_{\mf{c},j}^{-}(z)}^{\sigma_{\mf{c},j}^{+}(z)} \sigma \, \dd \sigma \\
  &\leq (\sigma_{\mf{c},j}^{+}(z))^2 |\mf{c}_j(z)|^2 
  \leq (\theta - (1- \varepsilon))^2 
  \lesssim_{\Theta} 1
  \end{aligned}
\end{equation*}
by \eqref{eq:sigmapm-bounds}.
\begin{equation*}
  \begin{aligned}
    \int_{\sigma_{\mf{c},j}^{-}(z)}^{\sigma_{\mf{c},j}^{+}(z)} \max(3 + \theta - \sigma \mf{c}_{j+1}, 0)^2 \, \frac{\dd\sigma}{\sigma}
    &\lesssim_{|\Theta|} \int_{\theta - (1 - \varepsilon)}^{3 + \theta} \, \frac{\dd\sigma}{\sigma}
    &= \Big| \log \Big(\frac{3 + \theta}{\theta - (1-\epsilon)} \Big) \Big|.
  \end{aligned}
\end{equation*}
The function on the right hand side is square-integrable over $\{\theta \in \Theta : \theta > 1-\epsilon\}$, so we're done.

\subsection{Bounding the third error term $\Err^{(3)}$}\label{sec:Err3}

We are estimating the function $\Err^{(3)}$ given by
\begin{equation*}
  \Err^{(3)}(z) := \sup_{\sigma}\Big( \sum_{j \in \N} \| (H_{\sigma_{j}^{+}}(z) + H_1 \ast \Upsilon_{\sigma_{j}^{+}}(z)) - (H_{\sigma_{j}^{-}}(z) + H_1 \ast \Upsilon_{\sigma_{j}^{-}}(z)) \|_{Y^*}^r \Big)^{1/r}
\end{equation*}
for all $z \in \R$, with supremum taken over all sequences $\sigma$ satisfying \eqref{eq:sigma-conditions}.
Fix such a sequence $\sigma(z)$ for each $z$, let $\psi = \Psi^{(0,\infty)}_{(\theta\sigma^{-1},\sigma)} = \Psi^{(0,\infty)}_{(\theta,1)}$ (using frequency-scale compatibility), and $\psi_{\sigma} := \Dil_{\sigma} \psi$, and write
\begin{equation}\label{eq:Err3-rewriting}
  \begin{aligned}
    & (H_{\sigma_{j}^+(z)}(z) + H_1 \ast \Upsilon_{\sigma_{j}^{+}(z)}(z)) - (H_{\sigma_{j}^{-}(z)}(z) + H_1 \ast \Upsilon_{\sigma_{j}^{-}(z)}(z)) \\
   &= \int_{\sigma_{j}^{-}(z)}^{\sigma_{j}^{+}(z)}  h(\cdot,\sigma) \ast \psi_{\sigma} (z) \, \frac{\dd \sigma}{\sigma} + \int_0^1 \big( h(\cdot,\sigma) \ast \psi_{\sigma} \ast (\Upsilon_{\sigma^{+}_{j}(z)} - \Upsilon_{\sigma^{-}_{j}(z)}) \big) (z)  \, \frac{\dd \sigma}{\sigma} \\
   &= \int_{\sigma_{j}^{-}(z)}^{\sigma_{j}^{+}(z)}  h(\cdot,\sigma) \ast \psi_\sigma (z) \, \frac{\dd \sigma}{\sigma} + \int_{\frac{\theta - \mf{b}}{\theta + \mf{b}} \sigma_{j}^{-}(z)}^{\frac{\theta + \mf{b}}{\theta} \sigma_{j}^{+}(z)} h(\cdot,\sigma) \psi_\sigma \ast (\Upsilon_{\sigma^{+}_{j}(z)} - \Upsilon_{\sigma^{-}_{j}(z)})(z) \, \frac{\dd \sigma}{\sigma}
  \end{aligned}
\end{equation}
using that $\psi_{\sigma} \ast (\Upsilon_{\sigma^{+}_{j}(z)} - \Upsilon_{\sigma^{-}_{j}(z)}) = 0$ for $\sigma \notin (\frac{\theta - \mf{b}}{\theta + \mf{b}} \sigma_{j}^{-}(z), \frac{\theta+\mf{b}}{\theta} \sigma_{j}^{+}(z))$.
Let $N_1$ denote the set of $j \in \N$ such that
\begin{equation*}
  \frac{\theta + \mf{b}}{\theta} \sigma^{-}_{j}(z) > \frac{\theta - \mf{b}}{\theta + \mf{b}} \sigma^{+}_{j}(z)
\end{equation*}
and let $N_2 := \N \sm N_1$.
Write
\begin{equation*}
  \Err^{(3)}_{\sigma} (z) \leq \Err^{(3,1)}_{\sigma}(z) + \Err^{(3,2)}_{\sigma}(z)
\end{equation*}
where the two functions on the right hand side have the same form as $\Err^{(3)}_{\sigma}(z)$ (meaning $\Err^{(3)}(z)$ with a particular choice of sequence $\sigma$ substituted in) but summing over only $N_1$ or $N_2$.
We will prove the two estimates
\begin{align}
  \label{eq:Err31-est} \|\Err^{(3,1)}_{\sigma}\|_{L^2(\R)} &\lesssim \|h\|_{L^2(B_1; \gamma_{\dd\sigma/\sigma}(0,1;Y^*))} \\ 
  \label{eq:Err32-est} \|\Err^{(3,2)}_{\sigma}\|_{L^2(\R)} &\lesssim \|h\|_{L^2(B_1; \gamma_{\dd\sigma/\sigma}(0,1;Y^*))} 
\end{align}
with implicit constants independent of $\theta$.
These will imply \eqref{eq:err3-term-bound}, and ultimately complete the proof of Theorem \eqref{thm:main-mass-dom}, which leads to the main theorem of this paper.

To estimate $\Err^{(3,1)}_{\sigma}$, for $j \in N_1$ and $z \in \R$ write
\begin{equation*}
  \begin{aligned}
    &(H_{\sigma_{j}^{+}(z)}(z) - H_1 \ast \Upsilon_{\sigma_{j}^{+}(z)}(z)) + (H_{\sigma_{j}^{-}(z)}(z) - H_1 \ast \Upsilon_{\sigma_{j}^{-}(z)}(z)) \\
    &\int_0^\infty h(\cdot,\sigma) \ast \Dil_\sigma P_{\sigma}^{j,z}(z) \, \frac{\dd \sigma}{\sigma}
  \end{aligned}
\end{equation*}
where
\begin{equation*}
  \begin{aligned}
  &\Dil_\sigma P_{\sigma}^{j,z}(w) \\
  &= \1_{(\sigma_{j}^{-}(z), \sigma_{j}^{+}(z))}(\sigma) \psi_{\sigma}(w) \\
  &\qquad + \1_{(\frac{\theta - \mf{b}}{\theta + \mf{b}} \sigma_{j}^{-} (z), \frac{\theta + \mf{b}}{\theta} \sigma_{j}^{+}(z) ) }(\sigma) \big[ \psi_{\sigma} \ast (\Upsilon_{\sigma_{j}^{+}(z)} - \Upsilon_{\sigma_{j}^{-}(z)}) \big](w).
  \end{aligned}
\end{equation*}
clearly we have that $\map{P^j}{\R \times \R_+}{\Sch(\RR)}$ is piecewise smooth.
Furthermore,
\begin{equation}\label{eq:Err3-claim1}
  \sup_{j,z,\sigma} \|P_{\sigma}^{j,z}\|_* \lesssim 1
\end{equation} 
for any Schwartz seminorm $\| \cdot \|_*$; this follows from a similar argument to the proof of \eqref{eq:Psigma-est},  using that $\sigma^{-}_{j}(z) \simeq \sigma^{+}_{j}(z)$ for $j \in N_1$, and that $P_{\sigma}^{j,z}$ vanishes unless $\sigma \simeq \sigma^\pm_{j}(z)$.
Taking $\lambda$ large enough, we can ensure that the intervals $(\frac{\theta - \mf{b}}{\theta + \mf{b}} \sigma_{j}^{-} (z), \frac{\theta + \mf{b}}{\theta} \sigma_{j}^{+}(z) )$ have finite overlap, so since $Y^*$ is UMD with cotype $r$, Lemma \ref{lem:short-variation} yields
\begin{equation*}
  \begin{aligned}
    \|\Err^{(3,1)}_\sigma\|_{L^2(\R)} 
    &\lesssim \|h\|_{L^2(B_1; \gamma_{\dd\sigma/\sigma}(0,1;Y^*))} \sup_{j,z} \|1\|_{L^2(\frac{\theta - \mf{b}}{\theta + \mf{b}} \sigma_{j}^{-} (z), \frac{\theta + \mf{b}}{\theta} \sigma_{j}^{+}(z) )} \\
    &\lesssim_\Theta \|h\|_{L^2(B_1; \gamma_{\dd\sigma/\sigma}(0,1;Y^*))}
  \end{aligned}
\end{equation*}
using that $\sigma^{-}_{j} \simeq \sigma^{+}_j$ for $j \in N_1$ in the last line.
This proves \eqref{eq:Err31-est}.

Finally we estimate $\Err^{(3,2)}_{\sigma}$.
For $z \in \R$ and $j \in N_2$, using that
\begin{equation*}
  \psi_{\sigma}= \psi_{\sigma} \ast \Upsilon_{\sigma^-_{j}(z)} 
\end{equation*}
for $\sigma > \frac{\theta + \mf{b}}{\theta}\sigma_{j}^{-}(z)$, 
the last line of \eqref{eq:Err3-rewriting} is equal to
\begin{equation*}
  \begin{aligned}
    &\int_{\sigma_{j}^{-}(z)}^{\sigma_{j}^{+}(z)}  h(\cdot,\sigma) \ast \psi_\sigma (z) \, \frac{\dd \sigma}{\sigma} + \int_{\frac{\theta - \mf{b}}{\theta + \mf{b}} \sigma_{j}^{-}(z)}^{\frac{\theta + \mf{b}}{\theta} \sigma_{j}^{+}(z)} h(\cdot,\sigma) \psi_\sigma \ast (\Upsilon_{\sigma^{+}_{j}(z)} - \Upsilon_{\sigma^{-}_{j}(z)})(z) \, \frac{\dd \sigma}{\sigma}
    \\
    &= \int_{\sigma_{j}^{-}(z)}^{\frac{\theta + \mf{b}}{\theta} \sigma_{j}^{-}(z)} h(\cdot,\sigma)  \ast \psi_{\sigma}(z) \, \frac{\dd \sigma}{\sigma}
    + \int_{\frac{\theta - \mf{b}}{\theta + \mf{b}} \sigma_{j}^{-}(z)}^{\frac{\theta + \mf{b}}{\theta} \sigma_{j}^{+}(z)} \big( h(\cdot,\sigma) \ast \psi_{\sigma} \ast \Upsilon_{\sigma^{+}_{j}(z)} \big) (z) \, \frac{\dd \sigma}{\sigma}
    \\
     &\qquad -\int_{\frac{\theta - \mf{b}}{\theta + \mf{b}} \sigma_{j}^-(z)}^{\frac{\theta + b}{\theta} \sigma_{j}^-(z)} \big( h(\cdot,\sigma) \ast \psi_{\sigma} \ast \Upsilon_{\sigma^-_{j}(z)} \big)(z) \, \frac{\dd \sigma}{\sigma}
     - \int_{\sigma_{j}^{+}(z)}^{\frac{\theta + \mf{b}}{\theta} \sigma_{j}^{+}(z)} \big( h(\cdot,\sigma) \ast \psi_{\sigma} \ast \Upsilon_{\sigma^{-}_{j}(z)} \big) (z) \, \frac{\dd \sigma}{\sigma}
     \\
     &= \int_{\sigma_{j}^{-}(z)}^{\frac{\theta + \mf{b}}{\theta} \sigma_{j}^{-}(z)} h(\cdot,\sigma)  \ast \psi_{\sigma}(z) \, \frac{\dd \sigma}{\sigma}
     + \int_{\frac{\theta - \mf{b}}{\theta + \mf{b}} \sigma_{j}^{+}(z)}^{\frac{\theta + \mf{b}}{\theta} \sigma_{j}^{+}(z)} \big( h(\cdot,\sigma) \ast \psi_{\sigma} \ast \Upsilon_{\sigma^{+}_{j}(z)} \big)(z) \, \frac{\dd \sigma}{\sigma} \\
     &\qquad -\int_{\frac{\theta - \mf{b}}{\theta + \mf{b}} \sigma_{j}^-(z)}^{\frac{\theta + b}{\theta} \sigma_{j}^-(z)} \big( h(\cdot,\sigma) \ast \psi_{\sigma} \ast \Upsilon_{\sigma^-_{j}(z)} \big)(z) \, \frac{\dd \sigma}{\sigma}
     - \int_{\sigma_{j}^{+}(z)}^{\frac{\theta + \mf{b}}{\theta} \sigma_{j}^{+}(z)} h(\cdot,\sigma) \ast \psi_{\sigma}(z) \, \frac{\dd \sigma}{\sigma} 
  \end{aligned}
\end{equation*}
We can thus express $\Err^{(3,2)}_{\sigma}$ as the sum of four functions $(\Err^{(3,2,i)}_{\sigma})_{i=1,4}$, each corresponding to the summands above.
Each of these can be bounded via Lemma \ref{lem:short-variation}, as we did for $\Err^{(3,1)}_{\sigma}$, yielding
\begin{equation*}
  \|\Err^{(3,2)}_{\sigma}\|_{L^2(\R)} \lesssim \|h\|_{L^2(B_1; \gamma_{\dd\sigma/\sigma}(0,1;Y^*))}
\end{equation*}
and completing the proof.


\section{A discussion of Banach function spaces}
\label{sec:BFS}

If $T \in \Lin(L^p(\R;\C))$ is a bounded linear operator, then for every Banach space $X$ one can consider the tensor extension $\map{T \otimes I_X}{L^p(\R;\C) \otimes X}{L^p(\R;\C) \otimes X}$, defined on elementary tensors by
\begin{equation*}
  (T \otimes I_X)(f \otimes x) := Tf \otimes x \qquad \forall f \in L^p(\R;\C), x \in X,
\end{equation*}
and extended by linearity to the whole algebraic tensor product $L^p(\R;\C) \otimes X$ (which is the linear span of elementary tensors, and not a Banach space).
This algebraic tensor product is dense in the Bochner space $L^p(\R;X)$, and if the bound
\begin{equation}
  \label{eq:extension}
  \|(T \otimes I_X)g\|_{L^p(\R;X)} \lesssim_X \|g\|_{L^p(\R;X)}
\end{equation}
holds for all $g \in L^p(\R;\C) \otimes X$, then $T \otimes I_X$ has a bounded extension to $L^p(\R;X)$.
The \emph{extension problem for $T$} asks for which $X$ the bound \eqref{eq:extension} holds.
This problem has been solved for many operators in harmonic analysis, most notably the Hilbert transform; the solution is that $X$ satisfies \eqref{eq:extension} (for all $p \in (1,\infty)$) if and only if $X$ is UMD.

The scalar-valued $r$-variational Carleson operator $\VCarl^{r}_{*,\C}$, i.e. the operator defined in \eqref{eq:var-carl-defn} in the case $X = \CC$, is bounded on $L^p(\R;\C)$ for all $p \in (1,\infty)$ provided that $r > 2$.
However, we cannot formulate the abstract extension problem detailed above, as $\VCarl^{r}_{*,\C}$ is not linear---only sublinear.
This is not because no extensions exist; on the contrary, when $X$ is a Banach function space there are two natural extensions of $\VCarl^r_{*,\C}$.
One can be bounded by the extrapolation methods of the first author, Lorist, and Veraar \cite{ALV19}, and the other is the operator $\VCarl^{r}_{*}$ considered above, which does not see the function space structure of $X$.
In this section we clarify the difference between these extensions, and compare the results of this article with those of \cite{ALV19}.
But first we have to introduce the language of Banach function spaces.

For a  measure space $\Omega$ (always assumed to be $\sigma$-finite), we let $\Sigma(\Omega)$ denote the vector space of simple functions $\Omega \to \C$ and $L^0(\Omega)$ the vector space of measurable functions $\Omega \to \C$ (modulo equality almost everywhere).

\begin{defn}
  Let $\Omega$ be a $\sigma$-finite measure space.
  A subspace $X$ of $L^0(\Omega)$, equipped with a norm $\|\cdot\|_X$, is called a \emph{Banach function space} (over $\Omega$) if it satisfies the following properties:
  \begin{itemize}
  \item if $x \in L^0(\Omega)$, $y \in X$, and $|x| \leq |y|$, then $x \in X$ and $\|x\|_X \leq \|y\|_X$;
  \item there exists $\zeta \in X$ with $\zeta > 0$;
  \item if $0 \leq x_n \uparrow x$ with $(x_n)_{n \in \N}$ a sequence in $X$, $x \in L^0(\Omega)$, and $\sup_{n \in \N} \|x_n\|_X < \infty$, then $x \in X$ and $\|x\|_X = \sup_{n \in \N} \|x_n\|_X$.
  \end{itemize}
\end{defn}

If $X$ is a Banach function space and $x \in X$, then the function $|x| \in X$ is well-defined.
It is also possible to take powers $|x|^p$, which are in $L^0(\Omega)$ but need not be in $X$.
Measuring such functions relates to the geometric concepts of convexity and concavity of Banach function spaces.
It also leads to the construction of new Banach function spaces, called \emph{concavifications}.

\begin{defn}
  Let $X$ be a Banach function space, $p \in [1,\infty]$, and $s \in (0,\infty)$.
  \begin{itemize}
  \item
    We say that $X$ is \emph{$p$-convex} if
    \begin{equation}
      \label{eq:convexity}
      \Big\| \Big( \sum_{j} |x_j|^p \Big)^{1/p} \Big\|_X \lesssim \Big( \sum_{j} \|x_j\|_X^p \Big)^{1/p}
    \end{equation}
    for all finite sequences $(x_j)$ in $X$, with the usual modification when $p = \infty$.
    We say that $X$ is \emph{$p$-concave} if the reverse estimate holds.
  \item
    We define the \emph{$s$-concavification} $X^s$ of $X$ by
    \begin{equation*}
      X^s := \{x \in L^0(\Omega) : |x|^{1/s} \in X \}
    \end{equation*}
    and equip it with the quasinorm
    \begin{equation*}
      \|x\|_{X^s} := \| |x|^{1/s} \|_X^s.
    \end{equation*}
  \end{itemize}
\end{defn}

Every Banach function space $X$ is $1$-convex and $\infty$-concave, and if $X$ is $p$-convex, $q$-concave, and infinite-dimensional, then $p \leq q$.
If $X$ is $p_0$-convex and $q_0$-concave, then it is $p$-convex and $q$-concave for all $p \in [1,p_0]$ and $q \in [q_0,\infty]$.
The Lebesgue space $L^r(\Omega)$ is both $r$-convex and $r$-concave.
If $X$ is $p$-convex the estimate \eqref{eq:convexity} also holds for infinite sequences; likewise if $X$ is $q$-concave.
The $s$-concavification $X^s$ is equivalent to a Banach space (i.e. $\|\cdot\|_{X^s}$ is equivalent to a norm) if and only if $X$ is $p$-convex for some $p \geq s$.
A key example is $L^p(\Omega)^s = L^{p/s}(\Omega)$.

Given a measure space $\Omega$, every function $\map{f}{\R}{L^0(\Omega)}$ can be identified with both a function $\map{f}{\R \times \Omega}{\C}$ (which we denote by the same letter) and a set of functions $\{f_\omega : \RR \to \CC : \omega \in \Omega\}$.\footnote{In this identification we ignore issues of measurability.}
In this notation we have
\begin{equation*}
  f(t)(\omega) = f(t,\omega) = f_\omega(t).
\end{equation*}
Now we can define two different $r$-variational Carleson operators on $X$-valued functions.

\begin{defn}
  Let $X$ be a Banach function space over a measure space $\Omega$, and let $r \in (2,\infty)$.
  \begin{itemize}
  \item For $f \in \Sch(\R;X)$ we define the \emph{norm $r$-variational Carleson operator} $\VCarl_{*,\mathrm{nm}}^{r} f \colon \R \to \C$ by
    \begin{equation*}
      \VCarl_{*,\mathrm{nm}}^{r} f(t) := \|\xi \mapsto \Carl_{\xi} f(t)\|_{V^r(\R;X)}.
    \end{equation*}
    This is the operator considered throughout the rest of the article, which does not see the Banach function space structure of $X$.
  \item
    For $f \in \Sch(\R;X)$ we define the \emph{pointwise $r$-variational Carleson operator} $\map{\VCarl_{*,\mathrm{pt}}^{r}f}{\R \times \Omega}{\C}$ by
    \begin{equation*}
      \VCarl_{*,\mathrm{pt}}^{r}f(t,\omega) := (\VCarl_{*,\C}^r f_\omega)(t) = \|\xi \mapsto \Carl_\xi f_\omega(t)\|_{V^r(\R;\C)}.
    \end{equation*}
  \end{itemize}
\end{defn}

The operator $\VCarl_{*,\mathrm{pt}}^{r}f$ is the \emph{lattice extension} of the operator $\VCarl_{*,\C}^{r}$ on scalar-valued functions; an analogous definition can be made for any operator (not necessarily linear) acting on scalar-valued functions.
We can equivalently write
\begin{equation*}
  \VCarl_{*,\mathrm{pt}}^{r}f(t) = \sup_{\mf{c} \in \Delta} \Big( \sum_{j \in \N} | \Carl_{\mf{c}_{j+1}}f(t) - \Carl_{\mf{c}_{j}}(t) |^r \Big)^{1/r},
\end{equation*}
with the supremum and $r$-th powers on the right-hand side being taken in the Banach function space $X$.
When $X = \C$ the operators $\VCarl_{*,\mathrm{nm}}$ and $\VCarl_{*,\mathrm{pt}}$ coincide, so both of these operators can be considered as extensions of $\VCarl_{*,\C}^{r}$.

In \cite[\textsection 5]{ALV19} the pointwise operator $\VCarl_{*,\mathrm{pt}}^{r}$ is considered (denoted there by $C_r$). 
Since this is the lattice extension of $\VCarl_{*,\C}^{r}$, and since $\VCarl_{*,\C}^{r}$ satisfies appropriate weighted estimates (proved in \cite{DPDU18}), the rescaled Rubio de Francia extrapolation theorem \cite[Corollary 3.6]{ALV19} implies the following bounds for the pointwise operator (stated as \cite[Theorem 5.2]{ALV19}).\footnote{The very recent results of \cite{LN20} can also be used to obtain these bounds directly from the sparse domination of $\VCarl_{*,\C}^r$ established in \cite{DPDU18}, bypassing the use of weighted estimates.}

\begin{thm}
  \label{thm:BFS-bounds}
  Suppose $r \in (2,\infty)$, and let $X$ be a Banach function space such that the $r'$-concavification $X^{r'}$ is UMD.
  Then for all $p \in (r',\infty)$ we have the estimates
  \begin{equation*}
    \|\VCarl_{*,\mathrm{pt}}^{r}f\|_{L^p(\R;X)}
    \lesssim_{X,p,r} \|f\|_{L^p(\R;X)} .
  \end{equation*}
\end{thm}

When $X$ is $r$-convex or $r$-concave, we have comparability between the two extensions.

\begin{prop}
  \label{prop:pt-nm-comp}
  Let $X$ be a Banach function space.
  If $X$ is $r$-convex then
  \begin{equation}
    \label{eq:pt-nm}
    \|\VCarl_{*,\mathrm{pt}}^r f(t)\|_X \leq |\VCarl_{*,\mathrm{nm}}^r f(t)| \qquad \forall f \in \Sch(\R;X), \, \forall t \in \R.
  \end{equation}
  The reverse estimate holds if $X$ is $r$-concave.
\end{prop}

\begin{proof}
  For all measurable functions $\map{\mf{c}}{\R}{\Delta}$ and all $t \in \R$ we have
  \begin{equation*}
    \Big\| \Big( \sum_{j \in \N} | \Carl_{\mf{c}_{j+1}(t)}f(t) - \Carl_{\mf{c}_{j}(t)}f(t) |^r \Big)^{1/r} \Big\|_X \leq \Big( \sum_{j \in \N}  \| \Carl_{\mf{c}_{j+1}(t)}f(t) - \Carl_{\mf{c}_{j}(t)}f(t) \|_X^r \Big)^{1/r}
  \end{equation*}
  if $X$ is $r$-convex, and the reverse estimate holds if $X$ is $r$-concave.
  Taking the supremum over all such $\mf{c}$ completes the proof.
\end{proof}

By combining Proposition \ref{prop:pt-nm-comp} with Theorems \ref{thm:BFS-bounds} and \ref{thm:main}, it is possible to obtain new conditions for boundedness of the two variational Carleson operators.
First let's obtain bounds for $\VCarl_{*,\mathrm{nm}}^{r}$.

\begin{thm}
  Let $r \in (2,\infty)$, and let $X$ be an $r$-concave Banach function space such that the $r'$-concavification $X^{r'}$ is UMD.
  Then for all $p \in (r',\infty)$,
  \begin{equation*}
    \|\VCarl_{*,\mathrm{nm}}^{r} f\|_{L^p(\R)} \lesssim_{p,r,X} \|f\|_{L^p(\R;X)} \qquad \forall f \in \Sch(\R;X).
  \end{equation*}
\end{thm}

The proof is a simple combination of the reverse estimate to \eqref{eq:pt-nm} and the bound from Theorem \ref{thm:BFS-bounds}.
The Banach function space $X = L^r(\Omega)$ satisfies the hypothesis of this theorem, but it does not satisfy those of Theorem \ref{thm:main}, as $X$ is not $r_0$-intermediate UMD for any $r_0 < r$.
Furthermore, for all $s \geq r$, the space $X = L^s(\Omega)$ satisfies the hypotheses of the theorem, and we get bounds for $\VCarl_{*,\mathrm{nm}}^{r}$ beyond the scope of Theorem \ref{thm:main} (as we get $L^p$ bounds for all $p > r'$, not just $p > (r/(s-1))'$ as in Theorem \ref{thm:main}).
Thus, when restricted to Banach function spaces, extrapolating from weighted scalar estimates is stronger than Theorem \ref{thm:main} (but of course, Theorem \ref{thm:main} applies to intermediate UMD spaces that have nothing to do with function spaces).

Here's what happens when we try to obtain bounds for $\VCarl_{*,\mathrm{pt}}^{r}$ in the same way.

\begin{thm}
  \label{thm:stupid}
  Let $r \in (2,\infty)$, and suppose that $X$ is an $r$-convex Banach function space.
  Suppose furthermore that $X$ is $r_0$-intermediate UMD for some $r_0 \in [2,r)$.
  Then
  \begin{equation*}
    \|\VCarl_{*,\mathrm{pt}}^{r}f\|_{L^p(\R;X)} \lesssim_{p,r,X} \|f\|_{L^p(\R;X)} \qquad \forall p \in (1,\infty).
  \end{equation*}
\end{thm}

As before, the proof is a simple combination of \eqref{eq:pt-nm} and Theorem \ref{thm:main}.
But there's a catch.
Suppose that $X$ is infinite-dimensional and satisfies the hypotheses of this theorem.
Then $X$ has cotype $r_0$, which by \cite[Corollary 1.f.9]{LT79} implies that $X$ is $q$-concave for all $q > r_0$.
Since $X$ is infinite-dimensional and $r$-convex, it follows that $r \leq q$ for all $q > r_0$, and hence that $r \leq r_0$.
But this contradicts the assumption that $r_0 < r$, so the only spaces $X$ satisfying the hypotheses of Theorem \ref{thm:stupid} are finite-dimensional, in which case everything reduces to the scalar case anyway.
Thus in the restricted context of Banach function spaces our results are weaker than the extrapolation results of \cite{ALV19}.
But for more general UMD Banach spaces such results do not apply, and at this point our result is the only one available.


\footnotesize
\bibliographystyle{amsplain}
\bibliography{bib/bibliography} 
\end{document}